\documentclass[12pt]{amsart}

\usepackage{DKstyle}

\newcommand{\vertiii}[1]{{\left\vert\kern-0.25ex\left\vert\kern-0.25ex\left\vert #1
    \right\vert\kern-0.25ex\right\vert\kern-0.25ex\right\vert}}
\theoremstyle{plain}

\newcommand{\fg}{\frak{g}}

\usepackage{caption}
\usepackage[labelfont=rm]{subcaption}

\usepackage[pagebackref=true, colorlinks]{hyperref}

\hypersetup{pdffitwindow=true,linkcolor=blue,citecolor=blue,urlcolor=blue,menucolor=blue}

\usepackage{comment}


\date{\today}%
\dedicatory{}%
\commby{}%

\title{The second moment of the Siegel transform in the space of symplectic lattices}
\author{Dubi Kelmer}
\address{Department of Mathematics, Boston College, Chestnut Hill MA 02467-3806, USA}
\email{kelmer@bc.edu}
\author{Shucheng Yu}
\address{Department of Mathematics, Boston College, Chestnut Hill MA 02467-3806, USA}
\email{shucheng.yu@bc.edu}
\thanks{This work was partially supported by NSF CAREER grant DMS-1651563.}

\begin{document}
\begin{abstract}
Using results from spectral theory of Eisenstein series, we prove a formula for the second moment of the Siegel transform when averaged over the subspace of symplectic lattices. This generalizes the classical formula of Rogers for the second moment in the full space of unimodular lattices. Using this new formula we give very strong bounds for the discrepancy of the number of lattice points in an Borel set, which hold for generic symplectic lattices.
\end{abstract}
\maketitle

\section*{Introduction}
Given a function $f:\R^n\to \C$ of sufficiently fast decay, its modified Siegel transform\footnote{In the standard Siegel transform the sum is over all lattice points.} $F_f$ is a function on the space, $X_n$, of unimodular lattices in $\R^n$, defined by
\begin{equation}
\label{e:Siegel} F_f(\Lambda)=\sum_{\vec{v}\in \Lambda_{\rm pr}}f(\vec{v}),\end{equation}
where $\Lambda_{\rm pr}$ denotes the set of primitive vectors in a lattice $\Lambda$. The space of lattices comes with a natural $\SL_n(\R)$-invariant probability measure $\nu_n$ and Siegel's mean value theorem states that
$$\int_{X_n}F_f(\Lambda)d\nu_n(\Lambda)=\frac{1}{\zeta(n)}\int_{\R^n}f(\vec{x})d\vec{x}.$$
In \cite{Rogers1955} Rogers proved beautiful and useful formulas for higher moments of $F_f$. In particular, his second moment formula states that, for $n\geq 3$
$$\int_{X_n}\left|F_f(\Lambda)\right|^2d\nu_n(\Lambda)=\frac{1}{\zeta(n)^2}\left|\int_{\R^n}f(\vec{x})d\vec{x}\right|^2+\frac{1}{\zeta(n)}\int_{\R^n}\left|f(\vec{x})\right|^2d\vec{x}+\frac{1}{\zeta(n)}\int_{\R^n}f(\vec{x})\overline{f(-\vec{x})}d\vec{x}.$$
Rogers' formula was since used in many applications of metric Diophantine approximations. For example,  Schmidt \cite{Schmidt1960} used it to study the lattice point counting problem, counting the number of lattice points in a Borel set $B\subseteq \R^n\bk\{0\}$, proving an optimal bound for the discrepancy for a generic lattice. The expected number of lattice points (respectively primitive lattice points) in $B$ is $\vol(B)$ (respectively $\frac{\vol(B)}{\zeta(n)}$) and the discrepancies are defined by
$$D(\Lambda,B)=\left|\tfrac{\#(\Lambda\cap B)}{\vol(B)}-1\right|, \quad D_{\rm pr}(\Lambda,B)=\left|\tfrac{\zeta(n)\#(\Lambda_{\rm pr}\cap B)}{\vol(B)}-1\right|.$$
For $n\geq 3$, applying Rogers' second moment formula to the indicator function of $B$ gives the mean square bound,
\begin{equation}\label{e:DiscVar}
\int_{X_n}|D_{\rm pr}(\Lambda,B)|^2d\nu_n(\Lambda)\leq \frac{2\zeta(n)}{\vol(B)}.
\end{equation}
Using this mean square bound, Schmidt \cite{Schmidt1960} showed that for any linearly ordered (with respect to inclusion) family, $\cB$, of Borel sets in $\R^n\bk\{0\}$, for $\nu_n$-a.e. lattice $\Lambda\in X_n$ there is a constant $C_{\Lambda}$ such that for all $B\in \cB$ with $\vol(B)>C_{\Lambda}$ both discrepancies $D_{\rm pr}(\Lambda,B)$ and $D(\Lambda,B)$ are bounded by $\tfrac{\log^2(\vol(B))}{\vol(B)^{1/2}}$.
%\begin{equation}\label{e:Schmidt}
%D_{\rm pr}(\Lambda,B)\leq\frac{\log^2(\vol(B))}{\sqrt{\vol(B)}},\quad D(\Lambda,B)\leq\frac{\log^2(\vol(B))}{\sqrt{\vol(B)}}.
%\end{equation}
\begin{rem}
We remark that the term $\log^2(\vol(B))$ above can be replaced by $\tfrac{\log(\vol(B))}{\psi^{1/2}(\log\vol(B))}$ for any positive, non-increasing function $\psi$ with $\int_1^\infty \psi(t)dt<\infty$, and we made a choice of $\psi(t)=1/t^2$ to simplify the statement. We further note that, while Schmidt stated his result for the larger space of all lattices (not just unimodular lattices) his proof actually implies the above stated result on unimodular lattices as well.
\end{rem}
\begin{rem}
While Rogers' second moment formula is valid only for $n\geq 3$, by using a different argument, Schmidt was able to prove a weaker variant %of \eqref{e:Schmidt}
also for $n=2$ after adding another factor of $\log(\vol(B))$.  Moreover,  it was later shown in \cite{AthreyaMargulis09}, using spectral theory of Eisenstein series of $\SL_2(\Z)$, that for primitive lattice points \eqref{e:DiscVar} still holds (up to a constant factor) also for $n=2$, and hence $D_{\rm pr}(\Lambda,B)\leq\tfrac{\log^2(\vol(B))}{\vol(B)^{1/2}}$ also holds for $\nu_2$-a.e. $\Lambda\in X_2$. \end{rem}
\begin{rem}
In addition to Schmidt's classical result, Rogers' formula was also used  in \cite{AthreyaMargulis09} to prove a random version of Minkowski theorem studying the set of lattices missing a large set,  and more recently in \cite{AthreyaMargulis18} to prove an effective and quantitative version of Oppenheim conjecture for generic forms. Rogers' formulas for higher moments were also used in \cite{SarnakStrombergsson06,Sodergren13} to study values of Epstein zeta functions.
\end{rem}

The starting point of our work is the observation that the Siegel transform can be interpreted as an incomplete Eisenstein series for a suitable maximal parabolic subgroup of $\SL_n(\R)$, and Rogers' formula can be seen as formula for moments of these incomplete Eisenstein series. It is thus possible to use the spectral theory of Eisenstein series to give another proof of Rogers' formula. The advantage of this approach is that it could be generalized to other semisimple groups and other maximal parabolic subgroups, leading  to other summation formulas. To demonstrate this approach, instead of giving another proof of Rogers' original formula, we shall prove a second moment formula with $\SL_n(\R)$ replaced by the symplectic group. We will then use this formula to give bounds for the discrepancy for a generic symplectic lattice, a problem recently studied by Athreya and Konstantoulas in \cite{AthreyaKonstantoulas2016}.

To describe the space of symplectic lattices, fix a symplectic form on $\R^{2n}$ and let $\Sp(2n,\R)$ denote the group of linear transformations preserving this form. The space of symplectic lattices in $\R^{2n}$ is then parameterized by the homogenous space $Y_n=\Sp(2n,\Z)\bk \Sp(2n,\R)$ via the map sending $g\in \Sp(2n,\R)$ to the lattice $\Lambda=\Z^{2n}g$, and it has a natural probability measure $\mu_n$ coming from Haar measure on $\Sp(2n,\R)$.  There is a natural generalization of Siegel's mean value theorem to the subspace of symplectic lattices, stating that
\begin{equation}\label{firstmomentin}
\int_{Y_n}F_f(\Lambda)d\mu_n(\Lambda)=\frac{1}{\zeta(2n)}\int_{\R^{2n}}f(\vec{x})d\vec{x}.
\end{equation}
Our main result is the following generalization of Rogers' second moment formula to this space. We say that a function $f$ on $\R^{2n}$ is even if $f(\vec{x})=f(-\vec{x})$ for almost every $\vec{x}\in\R^{2n}$ and denote by
$L^2_{\rm{even}}(\R^{2n})$ the space of even functions in $L^2(\R^{2n})$.
\begin{thm}\label{thm:Rogersformula}
For any $n\geq 1$, there is an isometry $\iota:L^2_{\rm{even}}(\R^{2n})\to L^2_{\rm{even}}(\R^{2n})$ such that for any bounded compactly supported even $f:\R^{2n}\to \C$
\begin{equation}\label{thm:formula}
\int_{Y_n}\left|F_f(\Lambda)\right|^2d\mu_n(\Lambda)=\left|\frac{1}{\zeta(2n)}\int_{\R^n} f(\vec{x})d\vec{x}\right|^2+\frac{2}{\zeta(2n)}\int_{\R^{2n}}\left(|f(\vec{x})|^2+ \overline{f(\vec{x})}\iota(f)(\vec{x})\right)d\vec{x}.
\end{equation}
\end{thm}
\begin{rem}
For $n=1$ we have that $\Sp(2n,\R)=\SL_2(\R)$ and this result follows from the analysis in \cite{KelmerMohammadi12}, we will thus concentrate below on the case when $n\geq 2$.
\end{rem}
\begin{rem}
The restriction to even function is just cosmetic. Any function can be decomposed as $f=f_{\rm{even}}+f_{\rm{odd}}$ and since lattices are invariant under inversion we always have that $F_f=F_{f_{\rm{even}}}$. We thus get a similar moment formula for all functions.
The condition that $f$ being compactly supported can also be relaxed, and replaced with the condition that $f$ decays sufficiently fast to insure that the series defining $F_f$ absolutely converges, or alternatively, with the condition that $f\in L^2(\R^{2n})$ is bounded and nonnegative. In particular, it applies for $f$ the indicator function of any finite-volume set in $\R^{2n}$.
\end{rem}

%\begin{rem}
%This formula is a special case of an inner product formulas between incomplete Eisenstein series. While such formulas are known in great generality (see e.g. \cite[Chapter II.2]{MoeglinWaldspurger95}), it is not always easy to translate the general results to an explicit formula such as\eqref{e:thm:formula}.  Instead we will give an a completely classical and self contained proof of this formula in our setting.
%\end{rem}

We now describe an application of this formula to the problem of symplectic lattice point counting. The space of symplectic lattices, $Y_n$, naturally sits inside the space, $X_{2n}$, of all unimodular lattices. But since it has $\nu_{2n}$-measure zero,  one can not say anything about symplectic lattices from statements on generic lattices in $X_{2n}$.  Instead, it is more natural to work with the measure $\mu_n$ on $Y_n$ directly.
In \cite{AthreyaKonstantoulas2016}, Athreya and Konstantoulas asked if one can give good bounds for the discrepancy, $D_{\rm pr}(\Lambda,B)$, which hold for any Borel set $B$ and for $\mu_n$-almost every symplectic lattice. They actually studied the slightly larger space $\tilde{Y}_n=G\Sp_n(\Z)\bk G\Sp_n(\R)$ of general symplectic lattices. In this space, under some mild geometric assumptions on the target sets, they managed to prove a power saving bound for the mean square of the discrepancy, as well as for the discrepancy of a generic general symplectic lattice.

As a first application of our moment formula we can get the following square root bound for the mean square of the discrepancy, in the primitive, as well as general lattice point counting problems.
 \begin{thm}\label{thm:meansquare}
For $n\geq 2$, for any Borel set $B\subset \R^{2n}\setminus\{0\}$ of finite volume
\begin{equation}\label{eq:measureestimate}
\int_{Y_n}|D_{\rm pr}(\Lambda,B)|^2d\mu_n(\Lambda)\leq \frac{4\zeta(2n)}{\vol(B)} ,
\end{equation}
and
\begin{equation}\label{eq:measureestimate2}
\int_{Y_n}|D(\Lambda,B)|^2 d\mu_n(\Lambda)\leq \frac{4\zeta(n)^2}{\zeta(2n)\vol(B)}.
\end{equation}
\end{thm}
\begin{rem}
When averaging over all unimodular lattices, for symmetric sets there is an equality, $\int_{X_n}|D_{\rm pr}(B,\Lambda)|^2d\nu_n(\Lambda)=\frac{2\zeta(n)}{\vol(B)}$, so the mean square bound is optimal.
For $Y_n$, however, we can only show an inequality. Since, we can not exclude additional cancelation between $\int |f(\vec{x})|^2d\vec{x}$ and $\int \overline{f(\vec{x})}\iota(f)(\vec{x})d\vec{x}$, it might be possible to construct a set $B$ for which the mean square is much smaller.
\end{rem}
Next, using this mean square bound we can adapt Schmidt's original arguments \cite{Schmidt1960}  to give the following square root bound for the discrepancy of a generic symplectic lattice.
\begin{thm}\label{thm:counting}
Let $n\geq 2$. For any linearly ordered family, $\cB$, of finite-volume Borel sets in $\R^{2n}\setminus\{0\}$, for $\mu_n$-a.e. $\Lambda\in Y_n$ there is $C_\Lambda$ such that for any $B\in\cB$ with $\vol(B)>C_{\Lambda}$
\begin{equation}\label{e:primitive}
D_{\rm pr}(\Lambda,B)\leq \frac{\log^2(\vol(B))}{\sqrt{ \vol(B)}}, \quad D(\Lambda,B)\leq \frac{\log^2(\vol(B))}{\sqrt{ \vol(B)}}.
\end{equation}
\end{thm}

%\begin{rem}
%With the exception of $n=1$ there is a standard way to translate results counting primitive vectors to counting all vectors.
%Explicitly, noting that any lattice point can be expressed as a unique multiple of a primitive lattice point, \eqref{thm:formula} implies that for $n\geq 2$
%$$\int_{Y_n}\left|\#(\Lambda\cap B)-\vol(B)\right|^2\leq \frac{4\zeta(n)^2\vol(B)}{\zeta(2n)}.$$
%With this mean square bound the same arguments show that for a.e. $\Lambda\in Y_n$ and for $B\in\cB$ with $\vol(B)$ sufficiently large
%$$|\#(\Lambda \cap B)-\vol(B)|\leq \sqrt{\vol(B)}\log^2(\vol(B)).$$
%We note that for $n=1$, this argument fails and the second moment of $|\#(\Lambda \cap B)-\vol(B)|$ over $X_2=Y_1$ is infinite. Nevertheless, Schmidt was still able to prove the bound for almost all lattices after adding another factor of $\log(\vol(B))$.
%\end{rem}

In many applications one is interested in a family of sets obtained by dilation of a fixed initial set, that is $\cB=\{tB:t>0\}$ for some fixed Borel set $B$.
If the set $B$ is not star shaped, (e.g., if $B$ is an annulus or even a ball not containing the origin) the family of dilations is not linearly ordered with respect to inclusion and Theorem \ref{thm:counting} does not apply directly.
Nevertheless, we show that the same result still applies also for dilations of sets that can be obtained as differences and unions of star shaped sets. Explicitly, we show
\begin{thm}\label{t:dilation}
Let $B\subseteq \R^{2n}\setminus\{0\}$ be a Borel set with $\vol(B)=1$. Assume that $B$ can be written as a finite disjoint union of sets $B_1,\ldots B_k$ with each $B_j=B_j^+\setminus B_j^-$ with $B_j^{-}\subseteq B_j^+$ star shaped. Then for $\mu_n$-a.e. $\Lambda\in Y_n$ there is a constant $C_\Lambda$ such that for all $t> C_\Lambda$ we have
$$D_{\rm pr}(\Lambda,tB)\leq \frac{\log^2(t)}{t^n},\quad D(\Lambda,tB)\leq \frac{\log^2(t)}{t^n}.$$
\end{thm}
\begin{rem}
We note that the same result (with the same proof) also holds for $\nu_n$-a.e. lattice in $X_n$. This is closely related to a recent result of Athreya and Margulis \cite{AthreyaMargulis18}  who proved a weaker power saving bound for this problem, but with no restriction on the set $B\subseteq \R^n$.
We further note that, since their proof relied on Rogers' formula, using our moment formula it could also be adapted to give the same result for $\mu_n$-a.e. lattice in $Y_n$.
\end{rem}
\subsection*{Acknowledgements}
The authors would like to thank Yuanqing Cai, Solomon Friedberg and Spencer Leslie for many helpful conversations.

%\begin{rem}
%The choice to work with primitive lattice points is mainly cosmetic. Noting that any lattice point can be expressed as a unique multiple of a primitive lattice point, it is possible to translate results on primitive lattice points to results on all lattice points when the dimension is greater than 2. In particular, for $n\geq 2$ the condition \eqref{e:primitive} for $\Lambda$ implies that
%$$|\#(\Lambda \cap B)-\vol(B)|\leq \zeta(n)C_\Lambda \log^2(\vol(B))\sqrt{\vol(B)},$$
%and a similar argument can be used to bound the second moment. We note that for $n=1$ this argument fails and the second moment of $|\#(\Lambda \cap B)-\vol(B)|$ over $X_2=Y_1$ is infinite.
%\end{rem}

\section{Preliminaries and notation}
\subsection{Notations}
For any integer $n\geq 1$, let $J_n$ denote the $n\times n$ exchange matrix, and let $G_n= \Sp(2n,\R)$ be the group of linear transformations preserving the symplectic form on $\R^{2n}$ with coordinate matrix
$\begin{pmatrix}
 & J_n\\
-J_n& \end{pmatrix}.$  We denote by $\G_n=\Sp(2n,\Z)$ the lattice of integer points, by $Y_n=\G_n\bk G_n$ and let $\mu_n$ denote the probability measure on $Y_n$ coming from the Haar measure of $G_n$. For any $F\in L^2(Y_n)$ and $f\in L^2(\R^{2n})$ we denote by $\|F\|_2$ and $\|f\|_2$ their $L^2$-norms with respect to $\mu_n$ and Lebesgue measure respectively. For $\vec{x}\in \R^{2n}$ we denote by $\|\vec{x}\|$ the standard Euclidean norm. We denote by $\R^+$ the set of positive real numbers.

\subsection{Coordinates on the symplectic group}
We will use the following coordinates on the symplectic group $G_n$.
Let $P_n\leq G_n$ denote the identity component of the maximal parabolic subgroup preserving the line spanned by $\vec{e}_{2n}=(0,\cdots,0,1)\in \R^{2n}$.

When $n=1$, $\Sp(2,\R)=\SL_2(\R)$ and $P_1$ has a Langlands decomposition $P_1=U_1M_1A_1$ with $M_1$ the trivial group and
$$U_1=\left\{u_t=\begin{pmatrix}
1 & t\\
0& 1\end{pmatrix}\ \big|\ t\in\R\right\},\ A_1= \left\{a_y=\begin{pmatrix}
y& 0 \\
0& y^{-1}\end{pmatrix}\ \big|\ y>0 \right\},$$
while for $n\geq 2$,
$P_n=U_nM_nA_n$ with
$$A_n= \left\{a_{y}=\begin{pmatrix}
y& 0 &0\\
0& I_{2n-2}& 0\\
0& 0& y^{-1}\end{pmatrix}\ \big|\ y>0 \right\},
$$
$$
M_n=\left\{ \widetilde{m}=\begin{pmatrix}
1& 0 & 0\\
0 & m &0 \\
0& 0& 1
\end{pmatrix}\ \big|\  m\in \Sp(2n-2,\R)\right\}$$
and
$$U_n=\left\{u_{\vec{t}}=\begin{pmatrix}
1 & \vec{t'} & t_{2n}\\
0 & I_{2n-2}& \vec{t'}^{\ast}\\
0& 0& 1\end{pmatrix}\ \big|\ \vec{t'}=(t_2,\ldots,t_{2n-1})\in \R^{2n-2}\ % {\textrm{and}}\ \vec{t}=(\vec{t'}, t_{2n})\in\R^{2n-1}
\right\}
,$$
where $\vec{t'}^{\ast}= (t_{2n-1}, \ldots,t_{n+1}, -t_n,\ldots, -t_2)^t$.
%For any $m\in \Sp(2n-2, \R)$, we denote
%$\widetilde{m}=\left(\begin{smallmatrix}
%1 &0 &0 \\
%0 &m &0 \\
%0 &0 & 1\end{smallmatrix}\right)\in M_n.$

Fix a maximal compact subgroup, $K_n=\Sp(2n,\R)\cap \SO(2n)$, of $\Sp(2n,\R)$ and  note that $K_n$ is isomorphic to the unitary group $U(n)$. We further note that for $n=1$,  $M_n\cap K_n$ is the trivial group, while for $n\geq 2$, $M_n\cap K_n$ is a maximal compact subgroup of $M_n\cong \Sp(2n-2,\R)$. Thus with slight abuse of notation, for any $n\geq 1$, we denote $K_{n-1}:=M_n\cap K_n$. For future reference, we note that the map sending $k\in K_n$ to its last row induces an identification between $K_{n-1}\bk K_n$ and the unit sphere $S^{2n-1}:=\{\vec{x}=(x_1,\ldots,x_{2n})\in\R^{2n}\ |\ \sum_{j=1}^{2n}x_j^2=1\}$.

\subsection{Haar measures}
We now wish to express the Haar measure $\mu_n$ explicitly in our coordinates.  To simplify notations, we denote by $Q_n=U_nM_n$. Let $\vec{e}_{2n}=(0,\ldots,0,1)\in\R^{2n}$. For any $g\in G_n$, note that $\vec{x}(g)=\vec{e}_{2n}g$ equals the last row of $g$. The map $G_n\to \R^{2n}$ sending $g$ to $\vec{e}_{2n}g$ gives an identification between the homogeneous space $Q_n\bk G_n$ and $\dot{\R}^{2n}:=\R^{2n}\bk \{0\}$. Let $d\vec{x}$ denote Lebesgue measure on $\R^{2n}$ that we think of also as a measure on $Q_n\bk G_n$ under this identification. We will also use the following polar coordinates. By Iwasawa decomposition any element in $Q_n\bk G_n$ can be represented uniquely by $a_yk$ with $a_y\in A_n$ and $k\in K_{n-1}\bk K_n\cong S^{2n-1}$. In these coordinates we have that
\begin{equation}\label{haarrelation}
d\vec{x}(a_yk)=\frac{2\pi^n}{\Gamma(n)}\frac{dy}{y^{2n+1}}d\sigma_n(k),
\end{equation}
where $\sigma_n$ denotes the right $K_n$-invariant probability measure on $K_{n-1}\bk K_n$ (which is the surface measure on the sphere).
%We note that $\frac{dy}{y^{2n+1}}d\sigma_n(k)$ is the unique (up to scalars) right $G_n$-invariant measure on $Q_n\bk G_n$.

Next, let $Q_n(\Z)=Q_n\cap \G_n$ and denote by $\mu_{Q_n}$ the probability Haar measure on $Q_n(\Z)\bk Q_n$. Explicitly, when $n=1$ we have that $Q_1=U_1$ and
$$d\mu_{Q_1}(q)=d\mu_{Q_1}(u_t)=dt,$$
while for $n\geq 2$, $M_n$ is isomorphic to $\Sp(2n-2,\R)$, and
\begin{equation}\label{Haar}
d\mu_{Q_n}(q)=d\mu_{Q_n}(u_{\vec{t}}\widetilde{m})=d\vec{t}d\mu_{n-1}(m).
\end{equation}
%where $d\vec{t}=\prod_{j=2}^{2n-1}dt_j$.
Using this decomposition we see that for any $F\in C_c^{\infty}(G_n)$,
\begin{equation}\label{pHaarmeasure}
\int_{G_n}F(g)d\mu_n(g)=\omega_n\int_{Q_n\bk G_n}\int_{Q_n}F(qa_yk)d\mu_{Q_n}(q)\frac{dy}{y^{2n+1}}d\sigma_n(k),
\end{equation}
for a suitable constant $\omega_n$.
%When $n=1$, for any $g=u_ta_yk$ with $u_t\in U_1, a_y\in A_1$ and $k\in \SO(2)$ as above, then under these coordinates $\mu_1$ is given by
%\begin{equation}
%d\mu_1(g)= \omega_1dt\frac{dy}{y^3}d\nu_1(k),
%\end{equation}
%where $\omega_1$ is some fixed positive constant to be determined. When $n\geq 2$, for any $g=u_{\vec{t}}\widetilde{m}a_yk$ with $u_{\vec{t}}\in U_n, a_y\in A_n, \widetilde{m}\in M_n$ and $k\in K_n$ as before, then $\mu_n$ can be defined inductively as following:
%%for any $f\in C_c^{\infty}(Sp(2n,\R))$, $\mu\mathcal C(\R^+)_n$ and $\mu_{n-1}$ satisfy the relation
%%\begin{equation\mathcal C(\R^+)}\label{Haar}
%%\int_{Sp(2n,\R)}f(g)d\mu_n(g)=\omega_n\int_{K_{n-1}\bk K_n}\int_{M_n}\int_{0}^{\infty}\int_{\R^{2n-1}}f(u_{\vec{t}}\widetilde{m}a_yk)d\vec{t}\frac{dy}{y^{2n+1}}d\mu_{n-1}(\widetilde{m})d\sigma_n(k),
%%\end{equation}
%\begin{equation}
%d\mu_n(g)=\omega_nd\vec{t}d\mu_{n-1}(\widetilde{m})\frac{dy}{y^{2n+1}}d\sigma_n(k),
%\end{equation}
%where $\omega_n$ is some fixed positive constant to be determined.
To compute $\omega_n$, after unfolding, $(\ref{firstmomentin})$ implies that
\begin{equation}\label{pHaarmeasure1}
d\mu_n(g)=\frac{1}{\zeta(2n)}d\mu_{Q_n}(q)d\vec{x}(a_yk),
\end{equation}
and comparing $(\ref{haarrelation}), (\ref{pHaarmeasure})$ and $(\ref{pHaarmeasure1})$ we get
\begin{equation}\label{Haarnorm}
\omega_n=\frac{2}{\xi(2n)},
\end{equation}
where $\xi(s)=\pi^{-s/2}\Gamma(\frac{s}{2})\zeta(s)$ is the Riemann Xi function.

%Throughout this paper, for any function $F$ on $Y_n$, we write $\|F\|_2$ to be the $L^2$-norm of $F$ with respect to the normalized Haar measure $\mu_n$, and for any $f$ on $Q_n\bk G_n\cong (\R^{2n})^{\times}$, we write $\|f\|_2$ to be the $L^2$-norm of $f$ with respect to the Lebesgue measure $d\vec{x}$. For any finite-volume Borel set $B\subset Q_n\bk G_n$ its volume is given by
%$$\textrm{vol}(B)=\int_{Q_n\bk G_n}\chi_B(\vec{x})d\vec{x},$$
%where $\chi_B$ is the characteristic function of $B$.

%{\color{red}{Need a suitable notation for the Lebesgue measure on $\R^{2n}$ and then write down the relation between Leb and $\frac{dy}{y^{2n+1}}d\sigma_n(k)$
%$$\textrm{Leb}=\frac{2\pi^n}{\Gamma(n)}\frac{dy}{y^{2n+1}}d\sigma_n(k)$$}}
\subsection{Mellin transform}
Let $C_c^{\infty}(\R^+)$ denote the space of smooth compactly supported functions on $\R^+$. For any $\rho\in C_c^{\infty}(\R^+)$, its Mellin transform is defined by
\begin{equation}\label{e:Mellin}
\hat{\rho}(s):=\int_0^{\infty}\rho(y)y^{-(s+1)}dy.
\end{equation}
We note for $\rho$ smooth and compactly supported, $\hat{\rho}(s)$ is an entire function with $|\hat{\rho}(s)|$ decaying super polynomially in $s$ as $\Im(s)\to \pm\infty$, and it satisfies the inversion formula (for any $\sigma$)
\begin{equation}\label{e:MellinInv}
\rho(y)=\frac{1}{2\pi i}\int_{(\sigma)}\hat{\rho}(s)y^sds.
\end{equation}
%For later use, we note that Mellin inversion theorem holds for more general functions. In particular, given $\rho\in C_c^{\infty}(\R^+)$ as above and a meromorphic function $L(s)$. If $L(s)$ is analytic in the half space $\Re(s)>\sigma_0$ for some $\sigma_0\in\R$ and $|L(s)|$ grows at most polynomially in $s$ as $\Im(s)\to\pm\infty$, then for any $\sigma> \sigma_0$ and $y>0$,
%$$\upsilon(y):=\frac{1}{2\pi i}\int_{(\sigma)}\hat{\rho}(s)L(s)y^sds$$
%is well-defined, and for $\Re(s)>\sigma_0$ we have $\hat{\upsilon}(s)=\hat{\rho}(s)L(s)$.

We will need to work with a slightly larger family of functions: We denote by $\mathcal C(\R^+)$, the family of functions $\rho$ with Mellin transform $\hat{\rho}(s)$ analytic for $\Re(s)>1$ with super polynomial decay on vertical strips, and note that these are all smooth functions with $\rho(y)\to 0$ as $y\to 0$. For $\rho\in\mathcal C(\R^+)$, we still have the inversion formula \eqref{e:MellinInv} for any $\sigma >1$.
%$C_c^{\infty}(\R^+)$ consists of functions $\rho$ on $\R^+$ such that its Mellin transform
%$$\hat{\rho}(s):=\int_0^{\infty}\rho(y)y^{-(s+1)}dy$$
%decays super polynomially as $\Im(s)\to \pm\infty$ and $\hat{\rho}(s)$ is analytic on the half plane $\Re(s)> 1$.
%Thus by the Mellin inversion theorem for any $\sigma>1$ the integral $\frac{1}{2\pi i}\int_{(\sigma)}\hat{\rho}(s)y^sds$ is well defined and equals $\rho(y)$. Moreover, we note that for any smooth compactly supported function $\rho$ on $\R^+$, $\hat{\rho}(s)$ decays super polynomially as $\Im(s)\to \pm\infty$ and is entire. Thus $C_c^{\infty}(\R^+)$ contains the space of smooth compactly supported functions on $\R^+$.
\subsection{Eisenstein series}
Using the Iwasawa decomposition $G_n= U_nM_nA_nK_n$,  for any $g\in G_n$ we write
\begin{equation}\label{idec}
g=u\widetilde{m}a_{y}k
\end{equation}
with $u\in U_n, \widetilde{m}\in M_n, a_y\in A_n$ and $k\in K_n$. Note that this decomposition is not unique, however $a_{y}$ is uniquely determined by (\ref{idec}) and $\frac{1}{y^2}$ is given by the sum of squares of the last row of $g$. Thus for each parameter $s\in \C$, the map
\begin{equation}\label{varphi}
\varphi_s: G_n\longrightarrow \C
\end{equation}
sending $g$ to $y^s$ is well-defined. %and given by $\vf_s(g)=\|\vec{e}_{2n}g\|^{-s}$ where $\|\cdot\|$ is the standard Euclidean norm on $\R^{2n}$.
Identifying $K_{n-1}\bk K_n$ with $Q_nA_n\bk G_n$, we can think of $\phi\in L^2(K_{n-1}\bk K_n)$ as a left $Q_nA_n$-invariant function on $G_n$. The corresponding Eisenstein series, $E_n(s,g,\phi)$, attached to this data is then defined as
$$E_n(s,g,\phi)=\sum_{\gamma\in \G_{P_n}\bk \G_n}\varphi_s(\gamma g)\phi(\gamma g),$$
where $\G_{P_n}=\G_n\cap P_n$. We note that $E_n(s,g,\phi)$ converges absolutely for $\Re(s)> 2n$ and has an analytic continuation and functional equation relating $s$ and $2n-s$. When $\phi=1$, the corresponding Eisenstein series $E_n(s,g,1)$ is right $K_n$-invariant, and we abbreviate it by $E_n(s,g)$.

Given a bounded compactly supported function $f$ on $Q_n\bk G_n$ (that we think of as a left $Q_n$-invariant function on $G_n$),
the corresponding incomplete Eisenstein series $\Theta_f\in L^2(Y_n,\mu_n)$ is defined by
$$\Theta_f(g)=\sum_{\gamma\in \G_{P_n}\bk \G_n}f(\gamma g).$$

For functions $f$ on $Q_n\bk G_n$ that factor as $f(a_yk)=\rho(y)\phi(k)$ with $\rho\in \mathcal C(\R^+)$ and $\phi\in L^2(K_{n-1}\bk K_n)$ the Eisenstein series and incomplete Eisenstein series can be related via the Mellin transform as follows: For $\sigma>2n$ we have
\begin{equation}\label{Mellin}
\Theta_f(g)= \frac{1}{2\pi i}\int_{(\sigma)}\hat{\rho}(s)E_n(s,g,\phi)ds.
\end{equation}
For future reference we record the following identity (essentially equivalent to \eqref{firstmomentin})
\begin{Lem}\label{firstmoment}
Let $\sigma_0<2n$ and let $\hat{\upsilon}$ be a meromorphic function that is analytic in the half plane $\Re(s)>\sigma_0$, and satisfies that $|\hat{\upsilon}(\sigma+ir)|\ll_{\sigma,k} (1+r^2)^{-k}$ for any $\sigma>\sigma_0$ and any $k\in \N$. Then for any
$\sigma>2n$, we have
$$\frac{1}{2\pi i}\int_{Y_n}\int_{(\sigma)}\hat{v}(s)E_n(s,g)dsd\mu_n(g)= \omega_n\hat{\upsilon}(2n),$$
with $\omega_n=\frac{2}{\xi(2n)}$ as before.
\end{Lem}
\begin{proof}
Let $\upsilon(y)=\frac{1}{2\pi i} \int_{(\sigma)}\hat{v}(s)y^sds$ and note that this does not depend on the choice of $\sigma>\sigma_0$, that $|\upsilon(y)|\ll_{\sigma} y^\sigma$ for any $\sigma>\sigma_0$ and that $\hat{\upsilon}(s)=\int_0^\infty \upsilon(y)y^{-s}\frac{dy}{y}$ for $\Re(s)>\sigma_0$. Now let $f(a_yk)=\upsilon(y)$ and let $\Theta_f(g)$  denote the corresponding incomplete Eisenstein series, then by \eqref{Mellin} for $\sigma>2n$
$$\int_{Y_n}\frac{1}{2\pi i}\int_{(\sigma)}\hat{\upsilon}(s)E_n(s,g)dsd\mu_n(g)=\int_{Y_n} \Theta_f(g)d\mu_n(g).$$
Next note that the condition $|f(a_yk)|\ll_{\sigma} y^{\sigma}$ with $\sigma>2n$ implies that the series defining $\Theta_f$ absolutely converges. Thus we can use unfolding and $\eqref{pHaarmeasure}$ to get that
\begin{align*}
\int_{Y_n}\Theta_{f}(g)d\mu_n(g)&=\omega_n\int_{Q_n\bk G_n}f(a_yk)\frac{dy}{y^{2n+1}}d\sigma_n(k)\int_{Q_n(\Z)\bk Q_n}d\mu_{Q_n}(q)\\
&=\omega_n\int_0^{\infty}\upsilon(y)y^{-(2n+1)}dy\\
&=\omega_n\hat{\upsilon}(2n).\qedhere
\end{align*}
\end{proof}

\section{Siegel transform and incomplete Eisenstein series}
As mentioned in the introduction the Siegel transform $F_f$ can be identified as an incomplete Eisenstein series corresponding to a maximal parabolic subgroup of $\SL_n(\R)$. We now show that it can also be identified as an incomplete Eisenstein series for the symplectic group and use it to prove some preliminary identities for the second moment.
\subsection{Transitivity}
The symplectic group acts transitively on $\dot{\R}^{2n}$ leading to the identification with $Q_n\bk G_n$. We now note that the
 integer points $\G_n$ also act transitively on $\Z_{\rm pr}^{2n}$ (the set of primitive points in $\Z^{2n}$), leading to the following bijection.
\begin{Lem}\label{coset}
There is a bijection between $\G_{P_n}\bk \G_n$ and $\Z_{\rm pr}^{2n}$ sending $\G_{P_n}\gamma\in \G_{P_n}\bk \G_n$ to the last row of $\gamma$.
\end{Lem}
\begin{proof}
We first show that $\G_{P_n}=Q_n(\Z)$. Indeed, for any $g\in\G_{P_n}=\G_n\cap P_n$, using Langlands decomposition $P_n=U_nM_nA_n$, we can write $g=u_{\vec{t}}\widetilde{m}a_y$; by direct computation it is easy to see that the $(1,1)$-entry of $g$ equals $y$ and the $(2n,2n)$-entry of $g$ equals $\frac1y$; but all entries are integers so $y=1$. Now for any $g\in\G_{P_n}=Q_n(\Z)$, $g\gamma$ has the same last row as $\gamma$. Hence the map $\G_{P_n}\bk \G_n\to \Z_{\rm pr}^{2n}$ sending $\G_{P_n}\gamma\in\G_{P_n}\bk \G_n$ to the last row of $\gamma$ is well-defined.

For injectivity, we need to show that if $\gamma_1,\gamma_2\in \G_n$ have the same last row, then $\gamma_1\gamma_2^{-1}\in \Gamma_{P_n}$. Since $\gamma_1\gamma_2^{-1}\in\G_n$, it suffices to show that $\gamma_1\gamma_2^{-1}\in Q_n$. Recall that for any $g\in G_n$ its last row is given by $\vec{e}_{2n}g$, and $Q_n$ consists exactly of matrices in $G_n$ with the last row equaling $\vec{e}_{2n}$. Hence it suffices to show that $\vec{e}_{2n}\gamma_1\gamma_2^{-1}=\vec{e}_{2n}$. But since $\gamma_1$ and $\gamma_2$ have the same last row, we have $\vec{e}_{2n}\gamma_1=\vec{e}_{2n}\gamma_2$ and the map is injective. For surjectivity, see \cite[Section 5.1]{MoskowitzSacksteder2010}.
\end{proof}

Using this bijection we get the following identification between the Siegel transform and the incomplete Eisenstein series.
\begin{Prop}\label{p:F2Theta}
For any function $f$ on $\R^{2n}$, let $\tilde{f}$ denote the corresponding function on $Q_n\bk G_n$ given by $\tilde{f}(g)=f(\vec{e}_{2n}g)$.
Then for any $g\in G_n$ and any symplectic lattice of the form $\Lambda=\Z^{2n}g$ we have that $F_f(\Lambda)=\Theta_{\tilde f}(g)$.
\end{Prop}
\begin{proof}
Since the primitive vectors in $\Lambda=\Z^{2n}g$ are exactly  $\Z^{2n}_{\rm pr}g$, by Lemma \ref{coset}, there is a unique $\gamma\in \G_{P_n}\bk \G_n$ such that any primitive $\vec{v}\in \Lambda$ is of the form
$\vec{v}=\vec{e}_{2n}\gamma g$. We thus get that
\begin{displaymath}
F_f(\Lambda)=\sum_{\vec{v}\in \Z^{2n}_{\rm pr}g}f(\vec{v})=\sum_{\g\in \G_{P_n}\bk \G_n}f(\vec{e}_{2n}\gamma g)=\Theta_{\tilde f}(g).\qedhere
\end{displaymath}
\end{proof}

With this identification we can rewrite the second moment as
$$\int_{Y_n}|F_f(\Lambda)|^2d\mu_n(\Lambda)=\|\Theta_{\tilde f}\|_2^2.$$
From here on we will work directly with a function $f$ on $Q_n\bk G_n$ and find a formula for $\|\Theta_{f}\|_2^2$.
\begin{rem}
The formula for $\|\Theta_f\|_2^2$ is a special case of more general inner product formula between incomplete Eisenstein Series. While such formulas are known in great generality (see e.g. \cite[Chapter II.2]{MoeglinWaldspurger95}), it is not always easy to translate the general results to an explicit formula such as \eqref{thm:formula}.  Instead we will give here a completely classical and self contained proof for a formula for $\|\Theta_{f}\|_2^2$, from which \eqref{thm:formula} easily follows.
\end{rem}

\subsection{Period formula}
Following the identification of the Siegel transform with the incomplete Eisenstein series we prove the following preliminary identity for the second moment in terms of an inner product with a certain period of the incomplete Eisenstein series.
\begin{Prop}\label{lem:pi2}
For any measurable, bounded and compactly supported function, $f$, on $Q_n\bk G_n$ we have
\begin{equation}\label{l2norm}
\|\Theta_f\|_2^2= \omega_n\int_{Q_n\bk G_n}\overline{f(a_yk)}\mathcal P_f(a_yk)   \frac{dy}{y^{2n+1}}d\sigma_n(k),
\end{equation}
where $\mathcal P_f$ is the period
\begin{equation}\label{e:Pf}
\mathcal P_f(a_yk)=\int_{M_n(\Z)\bk M_n}\int_{U_n(\Z)\bk U_n}\Theta_f(u_{\vec{t}}\widetilde{m}a_yk)d\vec{t}d\mu_{n-1}(m),
\end{equation}
with $M_n(\Z)=M_n\cap \G_n$ and $U_n(\Z)=U_n\cap \G_n$.
\end{Prop}
\begin{proof}
First, using unfolding we have
$$\|\Theta_f\|_2^2=\int_{\G_{P_n}\bk G_n}\overline{f(g)}\Theta_f(g)d\mu_n(g).$$
Next, recalling that $\G_{P_n}=Q_n(\Z)$, by $\eqref{Haar}$ and $\eqref{pHaarmeasure}$ we have
\begin{align*}
\|\Theta_f\|_2^2&= \omega_n\int_{Q_n\bk G_n}\int_{Q_n(\Z)\bk Q_n}\overline{f(qa_yk)}\Theta_f(qa_yk)d\mu_{Q_n}(q)\frac{dy}{y^{2n+1}}d\sigma_n(k)\\
&=\omega_n\int_{Q_n\bk G_n}\overline{f(a_yk)}\int_{M_n(\Z)\bk M_n}\int_{U_n(\Z)\bk U_n}\Theta_f(u_{\vec{t}}\widetilde{m}a_yk)d\vec{t}d\mu_{n-1}(m)\frac{dy}{y^{2n+1}}d\sigma_n(k), \end{align*}
where for the second equality we used the assumption that $f$ is left $Q_n$-invariant.
\end{proof}
\begin{rem}\label{r:fastdecay}
Here the condition that $f$ is compactly supported can be relaxed and replaced with the condition that $f(a_yk)\to 0$ sufficiently fast as $y\to 0$ to insure that the series defining $\Theta_f$ and the integrals in the formula all absolutely converge.
%In particular, we only need that $f(a_yk)$ remains bounded as $y\to\infty$ so this result can be applied to functions of the form $f(e_{2n}g)$ for $f(\vec{x})$ compactly supported on $\R^{2n}$.
\end{rem}
Having this preliminary identity, in order to prove our moment formula we need to compute the period $\mathcal P_f$ explicitly in terms of $f$. We will do it first in the special case where $f$ is spherical. We  then use raising operators to get it for other $K_n$-types and finally combine all $K_n$-types to get it for general functions.

% When $f$ is of the form $f(a_yk)=\rho(y)\phi(k)$, then by $(\ref{Mellin})$ for such $f$ we have
%\begin{equation}\label{preformula}
%\mathcal P_f(a_yk)=\frac{1}{2\pi i}\int_{M_n(\Z)\bk M_n}\int_{(\sigma)}\hat{\rho}(s)\int_{U_n(\Z)\bk U_n}E_n(s,u_{\vec{t}}\widetilde{m}a_yk,\phi)d\vec{t}dsd\mu_{n-1}(m).
%\end{equation}
%

\section{Period computation for spherical functions}
Let $f\in C^\infty(Q_n\bk G_n/K_n)$ be smooth and spherical in the sense that $f(a_yk)=\rho(y)$ with $\rho\in \cC(\R^+)$.  In this case by $(\ref{Mellin})$  we have for $\sigma>2n$
\begin{equation}\label{preformula}
\mathcal P_f(a_yk)=\frac{1}{2\pi i}\int_{M_n(\Z)\bk M_n}\int_{(\sigma)}\hat{\rho}(s)\int_{U_n(\Z)\bk U_n}E_n(s,u_{\vec{t}}\widetilde{m}a_y)d\vec{t}dsd\mu_{n-1}(m).
\end{equation}
Here the inner most integral is the constant term of the Eisenstein series  along the unipotent radical of $P_n$. We will first compute this constant term explicitly and then use it to compute the period.
%\begin{Prop}\label{intformula-s}
%For any spherical smooth function $f$ on $Q_n\bk G_n$, with $f(a_yk)= \rho(y)$ with $\rho\in \mathcal C(\R^+)$, we have
%\begin{equation}\label{ffs}
%\mathcal P_f(a_yk)=2\rho(y)+\frac{1}{\pi i}\int_{(n)}\hat{\rho}(s)\frac{\xi(s-2n+1)}{\xi(s)}y^{2n-s}ds+ \frac{2\hat{\rho}(2n)}{\xi(2n)}.
%\end{equation}
%\end{Prop}

\subsection{Constant term formula}
The constant term of $E_n(s,g)$ along the unipotent radical of $P_n$ is defined as
$$c_{E_n}(s,g):= \int_{U_n(\Z)\bk U_n}E_n(s,u_{\vec{t}}g)d\vec{t}.$$
For $n=1$ the formula $c_{E_1}(s,a_y)=2(y^s+\frac{\xi(s-1)}{\xi(s)}y^{2-s})$ is well known,
we now compute $c_{E_n}(s, \widetilde{m}a_y)$ explicitly for $n\geq 2$ in terms of the coordinates $(s,y,\widetilde{m})$.

%\begin{Lem}\label{coset}
%There is a bijection between $\G_{P_n}\bk \G_n$ and $\Z_{\rm pr}^{2n}$ sending $\G_{P_n}\gamma\in \G_{P_n}\bk \G_n$ to the bottom row of $\gamma$, where $\Z_{\rm pr}^{2n}$ is the set of primitive vectors in $\Z^{2n}$.
%\end{Lem}
%\begin{proof}
%We first note that since $\Gamma_{P_n}=Q_n(\Z)$, for any $g\in\G_{P_n}$, $g\gamma$ has the same bottom row as $\gamma$. Hence the map $\G_{P_n}\bk \G_n\to \Z_{\rm pr}^{2n}$ sending $\G_{P_n}\gamma\in\G_{P_n}\bk \G_n$ to the bottom row of $\gamma$ is well-defined.
%
%For injectivity, we need to show that if $\gamma_1,\gamma_2\in \G_n$ have the same bottom row, then $\gamma_1\gamma_2^{-1}\in \Gamma_{P_n}$. Since $\gamma_1\gamma_2^{-1}\in\G_n$, it suffices to show that $\gamma_1\gamma_2^{-1}\in Q_n$. Let $e_1=(0,\ldots,0,1)\in\R^{2n}$. Note that for any $g\in G_n$, $e_1g$ equals the bottom row of $g$, and $Q_n$ consists exactly of matrices in $G_n$ with the bottom row equaling $e_1$. Hence it suffices to show that $e_1\gamma_1\gamma_2^{-1}=e_1$. But since $\gamma_1$ and $\gamma_2$ have the same bottom row, we have $e_1\gamma_1=e_1\gamma_2$. This implies $e_1\gamma_1\gamma_2^{-1}=e_1$. This finishes the proof for injectivity. For surjectivity, see \cite[Section 5.1]{MoskowitzSacksteder2010}.
%\end{proof}
\begin{Prop}\label{const}
Let $n\geq 2$.
Keep the notation as above. Then
\begin{equation}\label{eq:constantformula}c_{E_n}(s, \widetilde{m}a_y)= 2y^s+ 2\frac{\xi(s-2n+1)}{\xi(s)}y^{2n-s}+ \frac{\xi(s-1)}{\xi(s)}yE_{n-1}(s-1,m),
\end{equation}
where $\xi(s)= \pi^{-s/2}\G(\frac{s}{2})\zeta(s)$ is the Riemann Xi function.
\end{Prop}

\begin{proof}
Using the expansion $\zeta(s)=\sum_{k=1}^{\infty}\frac{1}{k^s}$ for the Riemann zeta function and noting that any lattice point can be written in a unique way as a multiple of a primitive lattice point, Lemma \ref{coset} for $\Re(s)>2n$ implies that
\begin{equation}\label{comp}
\zeta(s)E_n(s,g)= \sum_{\vec{v}\in \Z^{2n}\bk\{0\}}\frac{1}{||\vec{v}g||^s}.
\end{equation}
%where for any $\vec{v}=(v_1,\ldots,v_{2n}), ||\vec{v}||^2=\sum_{i=1}^{2n}v_i^2$.
Next, for $\vec{v}=(v_1,\ldots,v_{2n})$, denote $\vec{v}^{M_n}=(v_2,\ldots, v_{2n-1})$, %Note that for any $k\in K_n$ and $g\in G_n$, $\|\vec{v}gk\|=\|\vec{v}g\|$.
and integrate \eqref{comp} over $U_n(\Z)\bk U_n$ to get that
\begin{align*}
\zeta(s)c_{E_n}(s, \widetilde{m}a_y)&=\int_{U_n(\Z)\bk U_n}\sum_{\vec{v}\in \Z^{2n}\bk\{0\}}\frac{d\vec{t}}{||\vec{v}u_{\vec{t}}\widetilde{m}a_y||^s}\\
&=\int_{[0,1]^{2n-1}} \left(\mathop{\sum_{(v_1,\vec{v}^{M_n})=0}}_{v_{2n}\neq 0}+\mathop{\sum_{v_1\neq 0}}_{(\vec{v}^{M_n},v_{2n})\in\Z^{2n-1}}+\mathop{\sum_{v_1=0, \vec{v}^{M_n}\neq 0 }}_{v_{2n}\in\Z}\right)\frac{d\vec{t}}{||\vec{v}u_{\vec{t}}\widetilde{m}a_y||^s}.
\end{align*}
By direct computation  we have
$$\vec{v}u_{\vec{t}}\widetilde{m}a_{y}=\left(v_1y, (v_1\vec{t'}+\vec{v}^{M_n})m, y^{-1}(v_1t_{2n}+\vec{v}^{M_n}\vec{t'}^{\ast}+v_{2n})\right),$$
where for $\vec{t}=(t_2,\ldots, t_{2n-1},t_{2n})\in\R^{2n-1}$, $\vec{t'}=(t_2,\ldots, t_{2n-1})\in\R^{2n-2}$.

First, when $(v_1,\vec{v}^{M_n})=0$, we have $\vec{v}u_{\vec{t}}\widetilde{m}a_y= (0,\ldots,0,y^{-1}v_{2n})$. Thus
$$\int_{[0,1]^{2n-1}}\mathop{\sum_{(v_1,\vec{v}^{M_n})=0}}_{v_{2n}\neq 0}\frac{d\vec{t}}{||\vec{v}u_{\vec{t}}\widetilde{m}a_y||^s}=\int_{[0,1]^{2n-1}}\sum_{v_{2n}\neq 0}\frac{y^s}{|v_{2n}|^{s}}d\vec{t}=2\zeta(s)y^s.$$

For the second term, when $v_1\neq 0$, we have
$$\vec{v}u_{\vec{t}}\widetilde{m}a_y= v_1\left(y,\ \left(\vec{t'}+\frac{\vec{v}^{M_n}}{v_1}\right)m,\ y^{-1}\left(t_{2n}+ \frac{\vec{v}^{M_n}\vec{t'}^{\ast}+v_{2n}}{v_1}\right)\right).$$
Thus
\begin{align*}
&\int_{[0,1]^{2n-1}}\mathop{\sum_{v_1\neq 0}}_{(\vec{v}^{M_n},v_{2n})\in\Z^{2n-1}}\frac{d\vec{t}}{||\vec{v}u_{\vec{t}}\widetilde{m}a_y||^s}\\
&= \int_{[0,1]^{2n-2}}\mathop{\sum_{v_1\neq 0}}_{\vec{v}^{M_n}\in\Z^{2n-2}}\int_0^1\sum_{v_{2n}\in\Z}\frac{dt_{2n}d\vec{t'}}{|v_1|^s\left(y^2+||(\vec{t'}+\frac{\vec{v}^{M_n}}{v_1})m||^2+ y^{-2}(t_{2n}+\frac{\vec{v}^{M_n}\vec{t'}^{\ast}+v_{2n}}{v_1})^2\right)^{\frac{s}{2}}}\\
&= \int_{[0,1]^{2n-2}}\sum_{v_1\neq 0}\sum_{\vec{v}^{M_n}\in\Z^{2n-2}}\int_{\R}\frac{dt_{2n}d\vec{t'}}{|v_1|^{s-1}\left(y^2+||(\vec{t'}+\frac{\vec{v}^{M_n}}{v_1})m||^2+ y^{-2}t_{2n}^2\right)^{\frac{s}{2}}}\\
&= \sum_{v_1\neq 0}\int_{\R^{2n-1}}\frac{d\vec{t}}{|v_1|^{s-2n+1}\left(y^2+||\vec{t'}m||^2+y^{-2}t_{2n}^2\right)^{\frac{s}{2}}}\\
&=2\zeta(s-2n+1)\int_{\R^{2n-1}}\frac{d\vec{t}}{\left(y^2+||\vec{t'}m||^2+y^{-2}t_{2n}^2\right)^{\frac{s}{2}}}.
\end{align*}
Since $m\in \Sp(2n-2,\R)$ has determinant one, the above integral equals
\begin{align*}
&2\zeta(s-2n+1)\int_{\R^{2n-1}}\frac{d\vec{t}}{\left(y^2+||\vec{t'}||^2+y^{-2}t_{2n}^2\right)^{\frac{s}{2}}}\\
&= 2\zeta(s-2n+1)y\int_{\R^{2n-1}}\frac{d\vec{t}}{\left(y^2+ t_2^2+\cdots+ t_{2n-1}^2+ t_{2n}^2\right)^{\frac{s}{2}}}\\
&=2\zeta(s-2n+1)y^{2n-s}\int_{\R^{2n-1}}\frac{d\vec{t}}{\left(1+t_2^2+\cdots+t_{2n-1}^2+t_{2n}^2\right)^{\frac{s}{2}}}\\
&=2\frac{\pi^{\frac{2n-1}{2}}\G(\frac{s-2n+1}{2})}{\G(\frac{s}{2})}\zeta(s-2n+1)y^{2n-s}.
\end{align*}
%where $\G(s)=\int_0^{\infty}e^{-t}t^{s-1}dt$ is the Gamma function.

Finally, when $v_1=0$ and $\vec{v}^{M_n}\neq 0$, say $v_{2n+1-i}\neq 0$ for some $2\leq i\leq 2n-1$, then we have
\begin{align*}
\vec{v}u_{\vec{t}}\widetilde{m}a_y&=\left(0,\ \vec{v}^{M_n}m,\ y^{-1}(\sum_{j=2}^{2n-1}\delta_jv_{2n+1-j}t_j+v_{2n})\right)\\
&=v_{2n+1-i}\left(0,\ \frac{\vec{v}^{M_n}m}{v_{2n+1-i}},\ y^{-1}\left(\delta_it_i+\mathop{\sum_{2\leq j\leq 2n-1}}_{j\neq i}\frac{\delta_jv_{2n+1-j}t_j}{v_{2n+1-i}}+ \frac{v_{2n}}{v_{2n+1-i}}\right)\right),
\end{align*}
where $\delta_j$ equals  $-1$ for $2\leq j\leq n$ and $1$ for $n+1\leq j\leq 2n-1$. Thus for such $\vec{v}^{M_n}$ we have
\begin{align*}
&\int_{[0,1]^{2n-1}}\sum_{v_{2n}\in\Z}\frac{d\vec{t}}{||\vec{v}u_{\vec{t}}\widetilde{m}a_y||^s}\\
&= \int_{[0,1]^{2n-3}}\int_0^1\sum_{v_{2n}\in\Z}\frac{dt_i\prod_{\substack{2\leq j\leq 2n-1 \\ j\neq i}}dt_j}{|v_{2n+1-i}|^s\left(||\frac{\vec{v}^{M_n}m}{v_{2n+1-i}}||^2+ y^{-2}\left(\delta_it_i+\sum_{\substack{2\leq j\leq 2n-1 \\ j\neq i}}\frac{\delta_jv_{2n+1-j}t_j}{v_{2n+1-i}}+ \frac{v_{2n}}{v_{2n+1-i}}\right)^2\right)^{\frac{s}{2}}}\\
&=\int_{[0,1]^{2n-3}}\int_{\R}\frac{dt_i\prod_{\substack{2\leq j\leq 2n-1 \\ j\neq i}}dt_j}{|v_{2n+1-i}|^{s-1}\left(||\frac{\vec{v}^{M_n}m}{v_{2n+1-i}}||^2+ y^{-2}t_i^2\right)^{\frac{s}{2}}}\\
&=\frac{y}{|v_{2n+1-i}|^{s-1}}\int_{\R}\frac{dt_i}{\left(||\frac{\vec{v}^{M_n}m}{v_{2n+1-i}}||^2+ t_i^2\right)^{\frac{s}{2}}}\\
&=y\frac{\sqrt{\pi}\G(\frac{s-1}{2})}{\G(\frac{s}{2})}\frac{1}{||\vec{v}^{M_n}m||^{s-1}}.
\end{align*}
Thus by $(\ref{comp})$ we have
\begin{align*}
\int_{[0,1]^{2n-1}}\mathop{\sum_{v_1=0, \vec{v}^{M_n}\neq 0 }}_{v_{2n}\in\Z}\frac{d\vec{t}}{||\vec{v}u_{\vec{t}}a_{y}m||^s}&=y\frac{\sqrt{\pi}\G(\frac{s-1}{2})}{\G(\frac{s}{2})}\sum_{\vec{v}^{M_n}\in\Z^{2n-2}\bk\{0\}} \frac{1}{||\vec{v}^{M_n}m||^{s-1}}\\
&=y\frac{\sqrt{\pi}\G(\frac{s-1}{2})}{\G(\frac{s}{2})}\zeta(s-1)E_{n-1}(s-1,m),
\end{align*}
where $E_{n-1}(s-1,m)$ is the Eisenstein series on $M_n\cong \Sp(2n-2,\R)$ defined as above. Combining these three terms and dividing both sides by $\zeta(s)$ we get \eqref{eq:constantformula} as claimed.
%$$c_{E_n}(s, \widetilde{m}a_y)= 2y^s+ 2\frac{\xi(s-2n+1)}{\xi(s)}y^{2n-s}+ \frac{\xi(s-1)}{\xi(s)}yE_{n-1}(s-1,m)$$
\end{proof}

\subsection{Period computation}
With the above constant term formula, we can compute $\mathcal P_f$ for  $f$ smooth and spherical as follows.
\begin{Prop}\label{intformula-s}
For any function $f$ on $Q_n\bk G_n$ such that $f(a_yk)= \rho(y)$ with $\rho\in \mathcal C(\R^+)$, we have
\begin{equation}\label{ffs}
\mathcal P_f(a_yk)=2\rho(y)+\frac{1}{\pi i}\int_{(n)}\hat{\rho}(s)\frac{\xi(s-2n+1)}{\xi(s)}y^{2n-s}ds+ \frac{2\hat{\rho}(2n)}{\xi(2n)}.
\end{equation}
\end{Prop}
\begin{proof}
Again, for $n=1$ this is well known and we will consider here the case of $n\geq 2$.
First we note that $\frac{\xi(s)}{\xi(s+1)}$ is analytic in the half plane $\Re(s)>1$ and $|\frac{\xi(s)}{\xi(s+1)}|$ grows polynomially in $s$ as $\Im(s)\to\pm\infty$. Thus $\hat{\rho}(s+1)\frac{\xi(s)}{\xi(s+1)}$ is analytic in the half plane $\Re(s)>1$ and $|\hat{\rho}(s+1)\frac{\xi(s)}{\xi(s+1)}|$ decays super polynomially in $s$ as $\Im(s)\to\pm\infty$. Thus it satisfies the conditions in Lemma \ref{firstmoment}.
%Hence for any $\sigma>2n$ and $y>0$
%$$\upsilon(y):=\frac{1}{2\pi i}\int_{(\sigma-1)}\hat{\rho}(s+1)\frac{\xi(s)}{\xi(s+1)}y^sds$$
%is well-defined with $\hat{\upsilon}(s)= \hat{\rho}(s+1)\frac{\xi(s)}{\xi(s+1)}$ for $\Re(s)> 1$.
Moreover, note that $M_n(\Z)\bk M_n=Y_{n-1}$ and for $\sigma> 2n$ we have $\sigma-1>2(n-1)$. Hence applying Lemma \ref{firstmoment} to $\hat{\rho}(s+1)\frac{\xi(s)}{\xi(s+1)}$ and recalling that $\omega_{n-1}=\frac{2}{\xi(2n-2)}$ we get
\begin{align*}
&\frac{1}{2\pi i}\int_{M_n(\Z)\bk M_n}\int_{(\sigma)}\hat{\rho}(s)\frac{\xi(s-1)}{\xi(s)}E_{n-1}(s-1,m)dsd\mu_{n-1}(m)\\
&=\frac{1}{2\pi i}\int_{Y_{n-1}}\int_{(\sigma-1)}\hat{\rho}(s+1)\frac{\xi(s)}{\xi(s+1)}E_{n-1}(s,m)dsd\mu_{n-1}(m)\\
%&=\frac{1}{2\pi i}\int_{X_{n-1}}\int_{(\sigma-1)}\hat{\upsilon}(s)E_{n-1}(s,m)dsd\mu_{n-1}(m)\\
&=\frac{2}{\xi(2n-2)}\hat{\rho}(2n-1)\frac{\xi(2n-2)}{\xi(2n-1)}=\frac{2\hat{\rho}(2n-1)}{\xi(2n-1)}.
\end{align*}
Thus by $(\ref{preformula})$, Proposition \ref{const} and the above equation, we have for $\sigma> 2n$,
\begin{align*}
\mathcal P_{f}(a_yk)&=\frac{1}{\pi i}\int_{(\sigma)}\hat{\rho}(s)\left(y^s+\frac{\xi(s-2n+1)}{\xi(s)}y^{2n-s}\right)ds+\frac{2\hat{\rho}(2n-1)}{\xi(2n-1)}y\\
&=2\rho(y)+\frac{1}{\pi i}\int_{(\sigma)}\hat{\rho}(s)\frac{\xi(s-2n+1)}{\xi(s)}y^{2n-s}ds+\frac{2\hat{\rho}(2n-1)}{\xi(2n-1)}y.
\end{align*}
Recall the Riemann Xi function $\xi(s)=\pi^{-s/2}\Gamma(\frac{s}{2})\zeta(s)$ is a meromorphic function with two simple poles at $s=0$ (with residue $-1$) and $s=1$ (with residue $1$) (see \cite[pp.170-171]{SteinShakarchi2003}). Moreover, for $n\geq 2$, $\xi(s)$ has neither zeros nor poles in the half plane $\Re(s)\geq n$. Thus in this half plane, the function $\hat{\rho}(s)\frac{\xi(s-2n+1)}{\xi(s)}y^{2n-s}$ has only two simple poles at $s=2n-1$ (with residue $\frac{-\hat{\rho}(2n-1)}{\xi(2n-1)}y$) and $s=2n$ (with residue $\frac{\hat{\rho}(2n)}{\xi(2n)}$). %The residue at $s=2n-1$ is $\frac{-\hat{\rho}(2n-1)}{\xi(2n-1)}y$ and the residue at $s=2n$ is $\frac{\hat{\rho}(2n)}{\xi(2n)}$.
Shifting the contour of integration from the line $\Re(s)=\sigma$ to the line $\Re(s)=n$ (picking up the contribution of the two poles) we get
\begin{align*}
\mathcal P_f(a_yk)&=2\rho(y)+\frac{1}{\pi i}\int_{(n)}\hat{\rho}(s)\frac{\xi(s-2n+1)}{\xi(s)}y^{2n-s}ds+\frac{2\hat{\rho}(2n)}{\xi(2n)}-\frac{2\hat{\rho}(2n-1)}{\xi(2n-1)}y+\frac{2\hat{\rho}(2n-1)}{\xi(2n-1)}y\\
&=2\rho(y)+\frac{1}{\pi i}\int_{(n)}\hat{\rho}(s)\frac{\xi(s-2n+1)}{\xi(s)}y^{2n-s}ds+ \frac{2\hat{\rho}(2n)}{\xi(2n)}.\qedhere
\end{align*}
%This completes the proof.
%\begin{align*}raising
%&\frac{1}{2\pi i}\int_{M_n(\Z)\bk M_n}\int_{(\sigma)}\hat{\rho}(s)\left(2y^s+2\frac{\xi(s-2n+1)}{\xi(s)}y^{2n-s}+ \frac{\xi(s-1)}{\xi(s)}yE_{n-1}(s-1,m)\right)dsd\mu_{n-1}(m)\\
%&=2\rho(y)+\frac{1}{\pi i}\int_{(\sigma)}\hat{\rho}(s)\frac{\xi(s-2n+1)}{\xi(s)}y^{2n-s}ds\\
%&+\frac{y}{2Z}\int_{M_n(\Z)\bk M_n}\hat{\rho}(s)\frac{\xi(s-1)}{\xi(s)}E_{n-1}(s-1,m)dsd\mu_{n-1}(m)
%\end{align*}
\end{proof}
%\begin{Rmk}
%We will fix $\eta=n$ later on when proving Theorem \ref{thm:Rogersformula}. However, for the convenience of the proof of Lemma \ref{oddcase} below, we keep the flexibility of the choice of $\eta$ at this moment.
%\end{Rmk}

\section{Raising operators}
When the function $f$ is not spherical, even if it is of a rather simple form $f(a_yk)=\rho(y)\phi(k)$,  we don't have such a nice formula for the constant term that will allow us to compute the period $\cP_f$ as we did for the spherical case. Instead of taking this direct approach, borrowing ideas from \cite{KelmerMohammadi12, Yu17},  we will start with the formula for $\cP_f$ for spherical $f$ and apply raising operators to obtain similar formulas for functions of the form $f(a_yk)=\rho(y)\phi(k)$ with $\phi$ of different $K_n$-types. Before we can do this we need some more background on $L^2(K_{n-1}\bk K_n)$ and how it decomposes into irreducible $K_n$-representations. We will assume here that $n\geq 2$ and note that for $n=1$ we have that $K_0$ is trivial, $K_1=\SO(2)$ and the decomposition is the standard Fourier decomposition.

\subsection{Root-space decomposition}
Let $\fg_n$ be the Lie algebra of $G_n$ and $\fk_n$ be the Lie algebra of $K_n$. Let $\fg_{n,\C}=\fg_n\otimes_{\R}\C$ and $\fk_{n,\C}:=\fk_n\otimes_{\R}\C$ be their complexifications respectively. Explicitly,
$$\fg_n=\left\{\begin{pmatrix}
A & B\\
C & D\end{pmatrix}\in \mathfrak{sl}_{2n}(\R)\ |\ A^tJ_n+ J_nD= 0, C^tJ_n= J_nC, B^tJ_n= J_n B\right\},$$
and
%Explicitly, if $A=\begin{pmatrix}a_{i,j}\end{pmatrix}_{n\times n}$, $B=\begin{pmatrix}b_{i,j}\end{pmatrix}_{n\times n}$, $C=\begin{pmatrix}c_{i,j}\end{pmatrix}_{n\times n}$ and $D=\begin{pmatrix}d_{i,j}\end{pmatrix}_{n\times n}$, then we have $a_{n+1-j,n+1-i}+ d_{i,j}=0, b_{i,j}=b_{n+1-j, n+1-i}$ and $c_{i,j}= c_{n+1-j,n+1-i}$ for any $1\leq i,j\leq n$. Let $K=Sp(2n,\R)\cap SO(2n)$ be a maximal compact subgroup of $Sp(2n,\R)$, then $K$ has Lie algebra
$$\fk_n= \left\{\begin{pmatrix}
A & B\\
-B^t & -J_nA^tJ_n\end{pmatrix}\in \mathfrak{sl}_{2n}(\R)\ |\ A^t+ A=0, B^tJ_n=J_nB\right\}.$$
Let $E_{\ell,j}$ be the $2n\times 2n$ matrix with one in the $(\ell,j)^{\rm{th}}$ entry and zero elsewhere. Let $\fh$ be the real vector space spanned by the set
$$\{E_{2n+1-j,j}-E_{j,2n+1-j}\ |\ 1\leq j\leq n\}.$$
Note that $\fh$ is the Lie algebra of a maximal torus of $K_n$. Let $\fh_{\C}$ be the complexification of $\fh$ and for each $1\leq j\leq n$, let
$$H_j:=i\left(E_{2n+1-j,j}-E_{j,2n+1-j}\right)\in \fh_{\C}$$
and let $\varepsilon_j: \fh_{\C}\to \C$ be the linear functional on $\fh_{\C}$ characterized by $\varepsilon_{\ell}(H_j)=\delta_{\ell j}$, where $\delta_{\ell j}$ is the Kronecker symbol. Then there is a root-space decomposition of $\fk_{n,\C}$ with respect to $\fh_{\C}$:
$$\fk_{n,\C}=\fh_{\C}\oplus\bigoplus_{\alpha\in \Phi\left(\fk_{n,\C},\mathfrak{h}\right)}\fk_{\alpha}$$
with $\Phi=\Phi(\fk_{n,\C},\fh_{\C})=\{\varepsilon_{\ell}-\varepsilon_j\ |\ 1\leq \ell\neq j\leq n\}$ the corresponding set of roots, and for any $\alpha\in \Phi$,
$$\fk_{\alpha}:=\left\{X\in \fk_{n,\C}\ |\ [H, X]=\alpha\left(H\right)X\ \textrm{for any $H\in\fh_{\C}$}\right\}$$
is the corresponding root-space. Fix a choice of simple roots
$$\{\varepsilon_j-\varepsilon_{j+1}\ |\ 1\leq j< n\}$$
such that the set of positive roots is given by $\Phi^+=\{\varepsilon_{\ell}-\varepsilon_j\ |\ 1\leq \ell< j\leq n\}$. For any finite-dimensional $\fk_{n,\C}$-module $V$ and any linear functional $\lambda: \fh_{\C}\to \C$, we say a nonzero vector $v\in V$ is of \textit{$K_n$-weight $\beta$} if $H\cdot v=\beta(H)v$ for any $H\in \fh_{\C}$ and we say $v\in V$ is a \textit{highest $K_n$-weight vector} if $X\cdot v=0$ for any $X\in \fk_{\alpha}$ with $\alpha\in \Phi^+$.

\subsection{Induced representations}
Let $P_n= U_nM_nA_n$, $Q_n=U_nM_n$ and $K_{n-1}=M_n\cap K_n$ as before. Let $\pi$ denote the right regular action of $G_n$ on functions on $Q_n\bk G_n$. For each parameter $s\in \C$, the induced representation, $I^s$, is the representation of $G_n$ consisting of measurable functions $f: Q_n\bk G_n\to \C$ satisfying
$$f(Q_na_yg)= y^sf(Q_ng)\ \textrm{for $\mu_n$-a.e. $g\in G_n$ and for any $a_y\in A_n$},$$
with $G_n$ acting on $I^s$ via the right regular action.
%For each parameter $s\in \C$, let $Z_s$ be the one-dimensional representation of $P_n$ such that for any $v\in \sigma_s$ and any $g=u\widetilde{m}a_y\in P_n$ with $u\in U_n, a_y\in A_n$ and $\widetilde{m}\in M_n$, $\sigma_s(g)\cdot v= y^s v$. Let $I_s:=\rm{Ind}_{P_n}^{G_n}\sigma_s$ be the corresponding induced representation. Explicitly, elements in $I_s$ are measurable functions $f: G_n\to \C$ satisfying
%\begin{equation}
%I^s:=\left\{f\in L^2(Q_n\bk G_n)\ |\ f(Q_na_yg)= y^sf(Q_ng)\ \textrm{for a.e. $g\in G_n$ and for any $a_y\in A_n$}\right\}.
%\end{equation}
By restricting to $K_n$, each $I^s$ can be viewed as a representation of $K_n$. Moreover, for each $s\in\C$, there is a natural isomorphism between $I^s$ and $L^2(K_{n-1}\bk K_n)$ as representations of $K_n$, given by the restriction map sending $f\in I^s$ to $f|_{K_{n-1}\bk K_n}$.

Recall the identification between $Q_n\bk G_n$ and $\dot{\R}^{2n}$ sending $Q_ng$ to $\vec{x}(g)\in\dot{\R}^{2n}$ with $\vec{x}(g)=(x_1,\ldots,x_{2n})=\vec{e}_{2n}g$ the last row of $g$. For $1\leq j\leq n$, let $z_j:=x_j+ ix_{2n+1-j}$ and $\bar{z}_j:=x_j- ix_{2n+1-j}$. Thus functions on $Q_n\bk G_n$ can be realized as functions in coordinates $(z_j,\bar{z_j})$. For any pair of nonnegative integers $(p,q)$, we say a polynomial $P$ in $(z_j,\bar{z_j})_{1\leq j\leq n}$ is \textit{bihomogeneous} of degree $(p,q)$ if
$$P(\lambda z_1,\ldots,\lambda z_n,\bar{\lambda} \bar{z}_1,\ldots,\bar{\lambda} \bar{z}_n)= \lambda^p\bar{\lambda}^qP(z_1,\ldots,z_n,\bar{z}_1,\ldots,\bar{z}_n)\ \textrm{for any $\lambda, z_i\in \C$}.$$
We say a polynomial is \textit{harmonic} if it is annihilated by the Euclidean Laplacian
$$\Delta:=4\sum_{j=1}^{n}\frac{\partial^2}{\partial z_j\partial\bar{z}_j},$$
where $\frac{\partial}{\partial z_j}:=\frac12(\frac{\partial}{\partial x_j}-i\frac{\partial}{\partial x_{2n+1-j}})$ and $\frac{\partial}{\partial \bar{z}_j}:=\frac12(\frac{\partial}{\partial x_j}+i\frac{\partial}{\partial x_{2n+1-j}})$. For each pair of nonnegative integers $(p,q)$, let $H^{p,q}$ be the space of bihomogeneous harmonic polynomials of degree $(p,q)$ and $\mathcal H^{p,q}:=\{P|_{K_{n-1}\bk K_n}  \ |\    P\in H^{p,q}\}$. As a function space, $L^2(K_{n-1}\bk K_n)$ decomposes as
$$L^2(K_{n-1}\bk K_n)=\hat{\bigoplus_{(p,q)\in \N^2}}\mathcal H^{p,q},$$
where $\hat{\bigoplus}$ denotes the Hilbert direct sum. Moreover, one can check that each $\mathcal H^{p,q}$ is invariant under the action of $K_n$. Thus this is a decomposition as $K_n$-representations. Correspondingly, for each parameter $s\in \C$, let $\mathcal H^{s,p,q}\subset I^s$ be the preimage of $\mathcal H^{p,q}$ under the restriction map from $I^s$ to $L^2(K_{n-1}\bk K_n)$ described above. Then $I^s$ has a corresponding decomposition (as $K_n$-representations)
$$I^s=\hat{\bigoplus_{(p,q)\in \N^2}}\mathcal H^{s,p,q}.$$
Moreover, there is a natural $\fg_{n,\C}$-module structure on $I^s_{\infty}:=\bigoplus_{(p,q)\in\N^2}\mathcal H^{s,p,q}$ by taking Lie derivatives: For any $f\in I_{\infty}^s$ and any $X\in\fg_n$, the Lie derivative $\pi(X)$, is defined by
%such that for any $X\in \fk_n$ and $f\in \mathcal H^{s,p,q}$,
$$\left(\pi(X)f\right)(g):=\frac{d}{dt}f(g\exp(tX))\bigg|_{t=0}.$$
This defines a $\fg_n$-module structure on $I_{\infty}^s$ and it induces a $\fg_{n,\C}$-module structure on $I_{\infty}^s$ via complexification: For any $X_1+iX_2\in\fg_{n,\C}$ with $X_1,X_2\in\fg_n$, define $\pi(X_1+iX_2):=\pi(X_1)+i\pi(X_2)$. Similarly, there is a $\fk_{n,\C}$-module structure on each $\mathcal H^{s,p,q}$.

For each $(s,p,q)$ as above, define $h_{s,p,q}$ by the formula
$$h_{s,p,q}(z_j,\bar{z}_j):=\frac{z_1^p\bar{z}_n^q}{\left(\sum_{j=1}^nz_j\bar{z}_j\right)^{\frac{s+p+q}{2}}}.$$
One can check that $h_{s,p,q}$ is an element in $\mathcal H^{s,p,q}$ and is of $K_n$-weight $p\varepsilon_1-q\varepsilon_n$. Moreover, using Weyl dimension formula for compact connected Lie groups (\cite[p.242]{BrockerDieck1995}) we can compute that the highest weight $\fk_{n,\C}$-module of weight $p\varepsilon_1-q\varepsilon_n$ is of dimension
\begin{equation}\label{dimension}
d_{p,q}:=\frac{(n+p-2)!(n+q-2)!(n+p+q-1)!}{(n-1)!(n-2)!p!q!}
\end{equation}
and it is exactly the dimension of $\mathcal H^{s,p,q}$ (see \cite{BezubikStrasburger2013}). Thus $\mathcal H^{s,p,q}$ is a highest weight $\fk_{n,\C}$-module with the highest $\fk_{n,\C}$-weight vector given by $h_{s,p,q}$.
\subsection{Raising operators}
Define the raising operators $\mathcal R^{(2,0)},\mathcal R^{(0,2)}\in\pi(\fg_{n,\C})$ by
$$\mathcal R^{(2,0)}:=\pi\left(\left(E_{1,1}-E_{2n,2n}\right)+i\left(E_{1,2n}+E_{2n,1}\right)\right)$$ and
$$\mathcal R^{(0,2)}:=\pi\left(\left(E_{n,n}-E_{n+1,n+1}\right)-i\left(E_{n,n+1}+E_{n+1,n}\right)\right).$$
%where $E_{\ell,j}$ is as before.
We will apply $\mathcal R^{(2,0)}$ and $\mathcal R^{(0,2)}$ to functions in $I^s_{\infty}$. By examining commutator relations, one can check that $\mathcal R^{(2,0)}$ and $\mathcal R^{(0,2)}$ commute (as left-invariant differential operators), and that $\mathcal R^{(2,0)}$ (resp. $\mathcal R^{(0,2)}$) sends the highest $K_n$-weight vector $h_{s,p,q}$ to a multiple of the highest $K_n$-weight vector $h_{s,p+2,q}$ (resp. $h_{s,p,q+2}$). Using the coordinates $(z_j,\bar{z}_j)$ on $Q_n\bk G_n$ as above, by direct computation we have
$$\mathcal R^{(2,0)}=2z_1\frac{\partial}{\partial\bar{z}_1}\quad \textrm{and}\quad \mathcal R^{(0,2)}=2\bar{z}_n\frac{\partial}{\partial z_n}.$$
%$$\mathcal R^{(2,0)}=\left(a_1+ib_1\right)\left(\frac{\partial}{\partial a_1}+i\frac{\partial}{\partial b_1}\right)$$
%and
%$$\mathcal R^{(0,2)}=\left(a_n-ib_n\right)\left(\frac{\partial}{\partial a_n}-i\frac{\partial}{\partial b_n}\right).$$
Thus for any $s\in\C$ and any nonnegative pair $(p,q)$, we have
\begin{equation}\label{ro1}
\mathcal R^{(2,0)}h_{s,p,q}=-(s+p+q)h_{s,p+2,q}
\end{equation}
and
\begin{equation}\label{ro2}
\mathcal R^{(0,2)}h_{s,p,q}=-(s+p+q)h_{s,p,q+2}.
\end{equation}
\subsection{Parity operator}
Using these raising operators and starting from a spherical function we can get functions of type $(p,q)$ for any even integers $(p,q)$. In order to be able to represent all even functions we also need functions of type $(p,q)$ with $p$ and $q$ both odd. For this reason we introduce the following auxiliary left-invariant differential operator $\mathcal R\in \pi(\fg_{n,\C})$ defined by
\begin{equation}\label{e:aux}
\mathcal R:=\pi\left(\left(E_{1,n}+E_{n,1}-E_{n+1,2n}-E_{2n,n+1}\right)+i\left(E_{1,n+1}+E_{n+1,1}+E_{n,2n}+ E_{2n,n}\right)\right).
\end{equation}
Using the  $(z_j,\bar{z}_j)$ coordinates on $Q_n\bk G_n$ this operator is given by
\begin{equation}\label{do3}
\mathcal R=2z_n\frac{\partial}{\partial \bar{z}_1}+2z_1\frac{\partial}{\partial \bar{z}_n}.
\end{equation}
%\begin{equation}\label{do3}
%\mathcal R=\left(a_n+ib_n\right)\left(\frac{\partial}{\partial a_1}+i\frac{\partial}{\partial b_1}\right)+\left(a_1+ib_1\right)\left(\frac{\partial}{\partial a_n}+i\frac{\partial}{\partial b_n}\right).
%\end{equation}
By examining commutator relations we see that $\mathcal R$ sends a vector of $K_n$-weight $\beta$ to a vector of $K_n$-weight $\beta+ \varepsilon_1+\varepsilon_n$, but $\mathcal R$ is not a raising operator in the sense that it does not send a highest $K_n$-weight vector in $I^s_{\infty}$ to another highest $K_n$-weight vector in $I^s_{\infty}$. Nevertheless, using $(\ref{do3})$ by direct computation we have
%\begin{equation}
%\mathcal Rh_{s,0,2}=-(s+2)\frac{(a_1+ib_1)(a_n+ib_n)(a_n-ib_n)^2}{\left(\sum_{j=1}^na_j^2+b_j^2\right)^{\frac{s+4}{2}}}
%\end{equation}
%\begin{align*}
%\mathcal Rh_{s,0,2}(z_j,\bar{z}_j)&=-2(s+2)\frac{(a_1+ib_1)(a_n+ib_n)(a_n-ib_n)^2}{\left(\sum_{j=1}^na_j^2+b_j^2\right)^{\frac{s+4}{2}}}\\
%&\quad+4\frac{(a_1+ib_1)(a_n-ib_n)}{\left(\sum_{j=1}^na_j^2+b_j^2\right)^{\frac{s+2}{2}}}.
%\end{align*}
\begin{align*}
\mathcal Rh_{s,0,2}(z_j,\bar{z}_j)&=\frac{-2(s+2)z_1z_n\bar{z}_n^2}{\left(\sum_{j=1}^nz_j\bar{z}_j\right)^{\frac{s+4}{2}}}+\frac{4z_1\bar{z}_n}{\left(\sum_{j=1}^nz_j\bar{z}_j\right)^{\frac{s+2}{2}}}\\
&=\frac{-2(s+2)\left(z_1z_n\bar{z}_n^2-\frac{2}{n+2}\left(\sum_{j=1}^nz_j\bar{z}_j\right)z_1\bar{z}_n\right)}{\left(\sum_{j=1}^nz_j\bar{z}_j\right)^{\frac{s+4}{2}}}+\frac{4(n-s)}{n+2}\frac{z_1\bar{z}_n}{\left(\sum_{j=1}^nz_j\bar{z}_j\right)^{\frac{s+2}{2}}}.
\end{align*}
In other words, using the polar coordinates $(a_y,k)$ on $Q_n\bk G_n$, we can write this as
\begin{equation}\label{ado4}
\mathcal Rh_{s,0,2}(a_yk)= (s+2)\varphi_s(a_y)\psi_{2,2}(k)+(n-s)\varphi_s(a_y)\psi_{1,1}(k),
\end{equation}
where
$\psi_{2,2}(z_j,\bar{z}_j)=-2\left(z_1z_n\bar{z}_n^2-\frac{2}{n+2}\left(\sum_{j=1}^nz_j\bar{z}_j\right)z_1\bar{z}_n\right)\big|_{K_{n-1}\bk K_n}\in \mathcal H^{2,2}$ and $\psi_{1,1}(z_j,\bar{z}_j)=\frac{4z_1\bar{z}_n}{n+2}\big|_{K_{n-1}\bk K_n}\in \mathcal H^{1,1}$.
\begin{rem}
Using this operator allows us to pass from a function of type $(0,2)$ (which can be obtained from a spherical function via the raising operator) to a function of type $(1,1)$. Then, starting from $\psi_{1,1}$ and applying the raising operators we can construct functions of type $(p,q)$ also for odd $p$ and $q$.
\end{rem}

\section{Period formula for non-spherical functions}
We can now use the raising operators to boot-strap the spherical period formula to obtain period formulas for other $K_n$-types.
For any nonnegative even integer $m$, we define
\[P_m(s):= \left\{
  \begin{array}{lr}
    1 & \textrm{if $m=0$}\\
    \prod_{j=0}^{\frac{m-2}{2}}\frac{2n-s+2j}{s+2j} & \textrm{if $m>0$,}
  \end{array}
\right.
\]
and $Z_m(s)=P_m(s)\frac{\xi(s-2n+1)}{\xi(s)}$. Note that, using the functional equation $\xi(s)=\xi(1-s)$ (\cite[p.170]{SteinShakarchi2003}) and the definition of $P_m(s)$, we see that $Z_m(s)$ satisfies the functional equation $Z_m(2n-s)Z_m(s)= 1$. In particular, when $\Re(s)=n$ we have
\begin{equation}\label{fees}
|Z_m(s)|^2=\overline{Z_m(s)}Z_m(s)=Z_m(\bar{s})Z_m(s)= Z_m(2n-s)Z_m(s)=1.
\end{equation}

For any pair of nonnegative integers $(p,q)$ with the same parity, let $\mathcal A^{p,q}$ denote the family of functions $f\in L^2(Q_n\bk G_n)$ of the form $f(a_yk)=\rho(y)\phi(k)$ with $\rho\in \mathcal C(\R^+)$ and $\phi\in \mathcal H^{p,q}$.
By using the raising operators we can obtain the following period formula for $f\in \cA^{p,q}$.
\begin{Prop}\label{prop:ef}
Let $(p,q)$ be a pair positive integers with the same parity. For any $f=\rho\phi\in \mathcal A^{p,q}$ we have
\begin{equation}\label{ffns}
\mathcal P_f(a_yk)=2\rho(y)\phi(k)+\frac{(-1)^p}{\pi i}\int_{(n)}\hat{\rho}(s)Z_{p+q}(s)y^{2n-s}ds\phi(k).
\end{equation}
%where $Z_{p+q}(s)=P_{p+q}(s)\frac{\xi(s-2n+1)}{\xi(s)}$ is defined as before.
\end{Prop}
\begin{rem}
When $n=1$ we have a similar formula for functions of the form $f(a_yk_\theta)=\rho(y)e^{i m\theta}$ with $m$ nonzero and even given by
$$
\cP_f(a_yk_\theta)=2\rho(y)e^{i m \theta}+\frac{1}{\pi i}\int_{(1)}\hat{\rho}(s)Z_{|m|}(s)y^{2-s}ds e^{i m\theta},
$$
where $k_{\theta}=\begin{pmatrix}
\cos\theta &\sin\theta\\
-\sin\theta &\cos\theta\end{pmatrix}$.
\end{rem}
\subsection{Preliminary lemmas}
We first prove the following two preliminary lemmas showing that having  \eqref{ffns} for one function implies that \eqref{ffns} holds for many functions.
\begin{Lem}\label{prelem1}
Let $(p,q)$ be as in Proposition \ref{prop:ef}. Suppose that there exists some $\rho\in \mathcal C(\R^+)$ and nonzero $\psi\in \mathcal H^{p,q}$ such that  \eqref{ffns} holds for $f=\rho\psi$, then \eqref{ffns} holds for $f=\rho\phi$ for any $\phi\in \mathcal H^{p,q}$.
\end{Lem}
\begin{proof}
Since $\mathcal H^{p,q}$ is an irreducible $\fk_{n,\C}$-module, for any $\phi\in \mathcal H^{p,q}$ there is some left-invariant differential operator $\mathcal D$ generated by $\pi(\fk_{n,\C})$ such that $\mathcal D \psi=\phi$. Recall the function $\varphi_s: G_n\to \C$ defined in $(\ref{varphi})$. Since $\mathcal D$ is generated by $\pi(\fk_{n,\C})$, $\mathcal D$ acts trivially on $\varphi_s$. Moreover, $\mathcal D$ commutes with the left regular action of $G_n$ on $\varphi_s\psi$. Thus $\mathcal D\varphi_s(\gamma g)\psi(\gamma g)= \varphi_s(\gamma g)\mathcal D\psi(\gamma g)=\varphi_s(\gamma g)\phi(\gamma g)$ for any $\gamma\in \G_n$ and $g\in G_n$. Here as before we view $\psi,\phi\in \mathcal H^{p,q}\subset L^2(K_{n-1}\bk K_n)$ as left $U_nA_nM_n$-invariant functions on $G_n$. Hence using $(\ref{Mellin})$ and applying $\mathcal D$ to $\Theta_{\rho\psi}$ we have for any $\sigma> 2n$,
\begin{align*}
\mathcal D \Theta_{\rho \psi}(g)&=\frac{1}{2\pi i}\int_{(\sigma)}\hat{\rho}(s)\sum_{\gamma\in \G_{P_n}\bk \G_n}\mathcal D\varphi_s(\gamma g)\psi(\gamma g)ds\\
&=\frac{1}{2\pi i}\int_{(\sigma)}\hat{\rho}(s)E_n(s,g,\phi)ds=\Theta_{\rho\phi}(g).
\end{align*}
We thus have
\begin{align*}
\mathcal D\mathcal P_{\rho\psi}(a_yk)&=\int_{M_n(\Z)\bk M_n}\int_{U_n(\Z)\bk U_n}\mathcal D\Theta_{\rho\psi}(u_{\vec{t}}\widetilde{m}a_yk)d\vec{t}d\mu_{n-1}(m)\\
&=\int_{M_n(\Z)\bk M_n}\int_{U_n(\Z)\bk U_n}\Theta_{\rho\phi}(u_{\vec{t}}\widetilde{m}a_yk)d\vec{t}d\mu_{n-1}(m)=\mathcal P_{\rho\phi}(a_yk).
\end{align*}
Since we assume  \eqref{ffns} holds for $\rho\psi$, we have
$$\mathcal P_{\rho \psi}(a_yk)=2\rho(y)\psi(k)+\frac{(-1)^p}{\pi i}\int_{(n)}\hat{\rho}(s)Z_{p+q}(s)y^{2n-s}ds\psi(k),$$
and applying $\mathcal D$ to both sides gives
\begin{align*}
\mathcal P_{\rho\phi}(a_yk)&=2\rho(y)\mathcal D\psi(k)+\frac{(-1)^p}{\pi i}\int_{(n)}\hat{\rho}(s)Z_{p+q}(s)y^{2n-s}ds\mathcal D\psi(k)\\
&=2\rho(y)\phi(k)+\frac{(-1)^p}{\pi i}\int_{(n)}\hat{\rho}(s)Z_{p+q}(s)y^{2n-s}ds\phi(k).\qedhere
\end{align*}
\end{proof}

\begin{Lem}\label{prelem2}
Fix a pair of positive integers $(p_0,q_0)$ with the same parity. Suppose \eqref{ffns} holds for all functions $f\in \mathcal A^{p_0,q_0}$, then for any pair of positive integers $(p,q)$ such that $p\geq p_0, q\geq q_0, p\equiv p_0 (\textrm{mod}\ 2)$ and $q\equiv q_0 (\textrm{mod}\ 2)$,  \eqref{ffns} holds for any functions $f\in\mathcal A^{p,q}$.
\end{Lem}
\begin{proof}
For any such $(p,q)$, let $h_{p,q}:=h_{s,p,q}|_{K_{n-1}\bk K_n}$ be the unique (up to scalars) highest $K_n$-weight vector in $\mathcal H^{p,q}$. In view of Lemma \ref{prelem1}, it suffices to show that \eqref{ffns} holds for any $\rho h_{p,q}$ with $\rho\in \mathcal C(\R^+)$. Fix $\sigma>2n$ and for any $y>0$ define
$$\upsilon(y):=\frac{1}{2\pi i}\int_{(\sigma)}\frac{\hat{\rho}(s)}{(s+p_0+q_0)(s+p_0+q_0+2)\cdots(s+p+q-2)}y^sds.$$
Then $\upsilon\in \cC(\R^+)$ with Mellin transform
\begin{equation}\label{invfor}
\hat{\upsilon}(s)= \frac{\hat{\rho}(s)}{(s+p_0+q_0)(s+p_0+q_0+2)\cdots(s+p+q-2)}
\end{equation}
for $\Re(s)>1$. Applying Proposition \ref{prop:ef} to $\upsilon h_{p_0,q_0}$ and using Mellin inversion for $v$, we have
\begin{equation}\label{pre-formula}
\mathcal P_{\upsilon h_{p_0,q_0}}(a_yk)=\frac{1}{\pi i}\int_{(n)}\hat{\upsilon}(s)\left(y^s+(-1)^{p_0}Z_{p_0+q_0}(s)y^{2n-s}\right)dsh_{p_0,q_0}(k).
\end{equation}
%\begin{align}
%&\int_{M_n(\Z)\bk M_n}\int_{U_n(\Z)\bk U_n}\Theta_{\upsilon}(u_{\vec{t}}\widetilde{m}a_yk)d\vec{t}d\mu_{n-1}(m) \label{lhs}\\
%&=\frac{1}{\pi i}\int_{(\sigma)}\hat{\upsilon}(s)y^sds+ \frac{1}{\pi i}\int_{(\sigma)}\hat{\upsilon}(s)\frac{\xi(s-2n+1)}{\xi(s)}y^{2n-s}ds+ y\hat{\upsilon}(2n-1). \label{rhs}
%\end{align}
Let $\mathcal D_{p,q}:=(-1)^{(p+q-p_0-q_0)/2}\left(\mathcal R^{(2,0)}\right)^{(p-p_0)/2}\left(\mathcal R^{(0,2)}\right)^{(q-q_0)/2}$. Note that $\varphi_sh_{p_0,q_0}=h_{s,p_0,q_0}$ where $\varphi_s$ is as in $(\ref{varphi})$. Thus $(\ref{ro1})$, $(\ref{ro2})$ and commutativity of $\mathcal R^{(2,0)}$ and $\mathcal R^{(0,2)}$ imply that
\begin{equation}\label{ro3}
\mathcal D_{p,q}\varphi_sh_{p_0,q_0}= (s+p_0+q_0)(s+p_0+q_0+2)\cdots (s+p+q-2)\varphi_sh_{p,q}.
\end{equation}
Moreover, since $\mathcal D_{p,q}$ commutes with the left regular action of $G_n$ on $\varphi_sh_{p_0,q_0}$, using $(\ref{Mellin})$, $(\ref{invfor})$ and $(\ref{ro3})$ we have for any $\sigma>2n$
\begin{align*}
\left(\mathcal D_{p,q} \Theta_{\upsilon h_{p_0,q_0}}\right)(g)&=\frac{1}{2\pi i}\int_{(\sigma)}\hat{\upsilon}(s)\sum_{\gamma\in \G_{P_n}\bk \G_n}\mathcal D_{p,q}\varphi_s(\gamma g)h_{p_0,q_0}(\gamma g)ds\\
%&=\frac{1}{2\pi i}\int_{(\sigma)}\hat{\upsilon}(s)(s+p_0+q_0)(s+p_0+q_0+2)\cdots(s+p+q-2)E_n(s,g,h_{p,q})ds\\
&=\frac{1}{2\pi i}\int_{(\sigma)}\hat{\rho}(s)E_n(s,g,h_{p,q})ds=\Theta_{\rho h_{p,q}}(g).
\end{align*}
Thus applying $\mathcal D_{p,q}$ to the left-hand side of $(\ref{pre-formula})$ we get
\begin{align*}
\mathcal D_{p,q}\mathcal P_{\upsilon h_{p_0,q_0}}(a_yk)&=\int_{M_n(\Z)\bk M_n}\int_{U_n(\Z)\bk U_n}\mathcal D_{p,q}\Theta_{\upsilon h_{p_0,q_0}}(u_{\vec{t}}\widetilde{m}a_yk)d\vec{t}d\mu_{n-1}(m)\\
&=\int_{M_n(\Z)\bk M_n}\int_{U_n(\Z)\bk U_n}\Theta_{\rho h_{p,q}}(u_{\vec{t}}\widetilde{m}a_yk)d\vec{t}d\mu_{n-1}(m)=\mathcal P_{\rho h_{p,q}}(a_yk).
\end{align*}
Similarly, using $(\ref{invfor})$ and $(\ref{ro3})$, after applying $\mathcal D_{p,q}$ the right-hand side of $(\ref{pre-formula})$ becomes
\begin{align*}
&\frac{1}{\pi i}\int_{(n)}\hat{\rho}(s)\left(y^s+(-1)^{p_0}Z_{p+q}(s)y^{2n-s}\right)dsh_{p,q}(k)\\
&=2\rho(y)h_{p,q}(k)+\frac{(-1)^{p}}{\pi i}\int_{(n)}\hat{\rho}(s)Z_{p+q}(s)y^{2n-s}dsh_{p,q}(k),
\end{align*}
%\begin{displaymath}
%\frac{1}{\pi i}\int_{(\sigma)}\hat{\rho}(s)\left(y^s+\frac{\xi(s-2n+1)}{\xi(s)}P_{p+q}(s)y^{2n-s}\right)dsh_{p,q}(k)+P_{p+q}(2n-1)\frac{\hat{\rho}(2n-1)}{\xi(2n-1)}yh_{p,q}(k).
%\end{displaymath}
%\begin{align*}
%%&\mathcal D_{p,q}\left(\frac{1}{\pi i}\int_{(\sigma)}\hat{\upsilon}(s)\left(y^s+P_{p_0+q_0}(s)\frac{\xi(s-2n+1)}{\xi(s)}y^{2n-s}\right)dsh_{p_0,q_0}(k)+ \hat{\upsilon}(2n-1)P_{p_0+q_0}(2n-1)\frac{\xi(2n-2)}{\xi(2n-1)}yh_{p_0,q_0}(k)\right)\\
%%&=\frac{1}{\pi i}\int_{(\sigma)}\hat{\upsilon}(s)\left(\mathcal D_{p,q}y^sh_{p_0,q_0}(k)+ \frac{\xi(s-2n+1)}{\xi(s)}\mathcal D_{p,q}y^{2n-s}h_{p_0,q_0}(k)\right)ds+ \hat{\upsilon}(2n-1)\frac{\xi(2n-2)}{\xi(2n-1)}\mathcal D_{p,q}yh_{p_0,q_0}(k)\\
%%&= \frac{1}{\pi i}\int_{(\sigma)}\hat{\rho}(s)y^sdsh_{p,q}(k)+ \frac{1}{\pi i}\int_{(\sigma)}\hat{\rho}(s)\frac{\xi(s-2n+1)}{\xi(s)}P_{p+q}(s)y^{2n-s}dsh_{p,q}(k)+ P_{p,q}(2n-1)y\hat{\upsilon}(2n-1)\\
%&\frac{1}{\pi i}\int_{(\sigma)}\hat{\rho}(s)\left(y^s+(-1)^pZ_{p+q}(s)y^{2n-s}\right)dsh_{p,q}(k)\\
%%&=2\rho(y)h_{p,q}(k)+ \frac{1}{\pi i}\int_{(\sigma)}\hat{\rho}(s)\frac{\xi(s-2n+1)}{\xi(s)}P_{p+q}(s)y^{2n-s}dsh_{p,q}(k)\\
%&\quad+(-1)^pP_{p+q}(2n-1)\frac{\hat{\rho}(2n-1)}{\xi(2n-1)}yh_{p,q}(k).
%\end{align*}
completing the proof.
\end{proof}
\subsection{Proof of Proposition \ref{prop:ef} for $p,q$ even.}
Starting from Proposition \ref{intformula-s} for spherical functions and applying $\cR^{(0,2)}$ and $\cR^{(2,0)}$ respectively, note that the constant term in \eqref{ffs} is killed by these two differential operators, to see that \eqref{ffns}
holds for any  $\rho h_{0,2}\in\cA^{0,2}$ and $\rho h_{2,0}\in\cA^{2,0}$ with $\rho\in \cC(\R^+)$. From this using Lemma \ref{prelem1} and Lemma \ref{prelem2} we see that Proposition \ref{prop:ef} holds for any even $p$ and $q$.
%\begin{Cor}\label{even}
%Let $(p,q)$ be a pair positive even integers. For any $f=\rho\phi\in \mathcal A^{p,q}$ we have
%$$\mathcal P_f(a_yk)=2\rho(y)\phi(k)+\frac{1}{\pi i}\int_{(n)}\hat{\rho}(s)Z_{p+q}(s)y^{2n-s}ds\phi(k).$$
%\end{Cor}
%\begin{proof}
%By Lemma \ref{prelem1} and Lemma \ref{prelem2}, it suffices to prove this formula for functions of the form $\rho h_{0,2}$ and $\rho h_{2,0}$ with $\rho\in \mathcal C(\R^+)$. This can be done by applying Proposition \ref{intformula-s} and the same arguments as in Lemma \ref{prelem2}, and noting that the raising operators $\mathcal R^{(2,0)}$ and $\mathcal R^{(0,2)}$ annihilate constants.
%\end{proof}
\subsection{Proof of Proposition \ref{prop:ef} for $p,q$ odd}
Let $\mathcal R$ denote the left invariant auxiliary operator defined in \eqref{e:aux}, and recall that
\begin{equation}\label{ado3}
\mathcal Rh_{s,0,2}(a_yk)= (s+2)\varphi_s(a_y)\psi_{2,2}(k)+(n-s)\varphi_s(a_y)\psi_{1,1}(k),
\end{equation}
with $\psi_{2,2}\in \mathcal H^{2,2}$ and $\psi_{1,1}\in \mathcal H^{1,1}$ given in \eqref{ado4}.
Now, in view of the Lemma \ref{prelem1} and Lemma  \ref{prelem2} to prove Proposition \ref{prop:ef} for all $p,q$ odd, it is enough to show that \eqref{ffns} holds for $f=\rho\psi_{1,1}\in \cA^{1,1}$ that we show as follows.

\begin{Lem}\label{oddcase}
Let $\psi_{1,1}\in\mathcal H^{1,1}$ be as above. For any $f\in \mathcal A^{1,1}$ of the form $f=\rho \psi_{1,1}$ with $\rho\in \mathcal C(\R^+)$ we have
$$\mathcal P_f(a_yk)=2\rho(y)\psi_{1,1}(k)-\frac{1}{\pi i}\int_{(n)}\hat{\rho}(s)Z_{2}(s)y^{2n-s}ds\psi_{1,1}(k).$$
%with $Z_2(s)=P_2(s)\frac{\xi(s-2n+1)}{\xi(s)}$ as before.
\end{Lem}
%\begin{Lem}\label{intformula-ns-odd}
%Let $(p,q)$ be a pair positive odd integers. For any $f\in L^2(Q_n\bk G_n)$ of the form $f(a_yk)=\rho(y)\phi(k)$ with $\rho\in C_c^{\infty}(\R^+)$ and $\phi\in\mathcal{H}^{p,q}$, and for any $\sigma>2n$, we have
%\begin{align*}
%\mathcal P_f(a_yk)&=\frac{1}{\pi i}\int_{(\sigma)}\hat{\rho}(s)\left(y^s+Z_{p+q}(s)y^{2n-s}\right)ds\phi(k)\\
%&\quad+P_{p+q}(2n-1)\frac{\hat{\rho}(2n-1)}{\xi(2n-1)}y\phi(k),
%\end{align*}
%where $P_{p+q}(s)$ is defined as above.
%\end{Lem}
\begin{proof}

For $f=\rho \psi_{1,1}$, fix $\sigma> 2n$ and for any $y>0$ define
$$\upsilon(y):=\frac{1}{2\pi i}\int_{(\sigma)}\frac{\hat{\rho}(s)}{n-s}y^sds.$$
As before $\upsilon$ is independent of the choice of $\sigma>2n$ and for $\Re(s)>n$ we have
\begin{equation}\label{mein}
\hat{\upsilon}(s)=\frac{\hat{\rho}(s)}{n-s}.
\end{equation}
We note that if $\hat{\rho}(n)\neq 0$, then $\frac{\hat{\rho}(s)}{n-s}$ has a pole at $s=n$. Thus $\upsilon h_{0,2}$ is not necessarily contained in $\mathcal A^{0,2}$. However, since $\hat{\upsilon}(s)$ is analytic in the half plane $\Re(s)>n$ and satisfies the Mellin inversion formula $v(y)=\frac{1}{2\pi i}\int_{(\sigma)}\hat{\upsilon}(s)y^sds$ for $\sigma>2n$, using the same arguments\footnote{The only difference is that in Proposition \ref{intformula-s}, instead of shifting the contour of integration from $\Re(s)=\sigma$ to $\Re(s)=n$, we shift the contour from $\Re(s)=\sigma$ to $\Re(s)=\eta$.} as in Proposition \ref{intformula-s}, Lemma \ref{prelem1} and Lemma \ref{prelem2}, one can deduce the formula
\begin{equation}\label{testid}
\mathcal P_{\upsilon h_{0,2}}(a_yk)=\frac{1}{\pi i}\int_{(\eta)}\hat{\upsilon}(s)\left(y^s+ Z_2(s)y^{2n-s}\right)dsh_{0,2}(k) \end{equation}
for  $\eta\in(n,2n-1)$. With similar computations as in Lemma \ref{prelem1} and \ref{prelem2}, using $(\ref{ado3})$, $(\ref{mein})$, and applying $\mathcal R$ to the left-hand side of $(\ref{testid})$ we get
\begin{equation}\label{lhsado}
\mathcal R\mathcal P_{\upsilon h_{0,2}}(a_yk)= \mathcal P_{\upsilon_1 \psi_{2,2}}(a_yk)+ \mathcal P_{\rho \psi_{1,1}}(a_yk),
\end{equation}
where $\upsilon_1(y):=\frac{1}{2\pi i}\int_{(\sigma)}\hat{\upsilon}(s)(s+2)y^sds$ for $y>0$ and $\sigma> 2n$. We note that for $\Re(s)> n$, $\hat{\upsilon}_1(s)=\hat{\upsilon}(s)(s+2)$ and hence for $\eta\in (n,2n-1)$ we have
\begin{equation}\label{1formula}
\mathcal P_{\upsilon_1 \psi_{2,2}}(a_yk)=\frac{1}{\pi i}\int_{(\eta)}\hat{\upsilon}(s)(s+2)\left(y^s+ Z_4(s)y^{2n-s}\right)ds\psi_{2,2}(k).
\end{equation}
On the other hand, using $(\ref{ado3})$ and applying $\mathcal R$ to the right-hand side of $(\ref{testid})$ we get
\begin{align*}
&\frac{1}{\pi i}\int_{(\eta)}\hat{\upsilon}(s)\left(\mathcal Ry^sh_{0,2}(k)+ Z_2(s)\mathcal Ry^{2n-s}h_{0,2}(k)\right)ds\\
&=\frac{1}{\pi i}\int_{(\eta)}\hat{\upsilon}(s)\bigg((s+2)\varphi_s(a_y)\psi_{2,2}(k)+(n-s)\varphi_s(a_y)\psi_{1,1}(k)\\
&+ Z_2(s)\left((2n-s+2)\varphi_{2n-s}(a_y)\psi_{2,2}(k)+(s-n)\varphi_{2n-s}(a_y)\psi_{1,1}(k)\right)\bigg)ds\\
&=\mathcal P_{\upsilon_1\psi_{2,2}}(a_yk)+\frac{1}{\pi i}\int_{(\eta)}\hat{\rho}(s)\left(y^s-Z_2(s)y^{2n-s}\right)ds\psi_{1,1}(k)\\
&=\mathcal P_{\upsilon_1\psi_{2,2}}(a_yk)+\frac{1}{\pi i}\int_{(n)}\hat{\rho}(s)\left(y^s-Z_2(s)y^{2n-s}\right)ds\psi_{1,1}(k)\\
&=\mathcal P_{\upsilon_1\psi_{2,2}}(a_yk)+2\rho(y)\psi_{1,1}(k)-\frac{1}{\pi i}\int_{(n)}\hat{\rho}(s)Z_2(s)y^{2n-s}ds\psi_{1,1}(k),
\end{align*}
where for the second equality we used the relations $(\ref{mein})$ and $\hat{\upsilon}_1(s)=\hat{\upsilon}(s)(s+2)$ for $\Re(s)=\eta> n$, and equation $(\ref{1formula})$, for the third equality we shift the contour of integration from $\Re(s)=\eta$ to $\Re(s)=n$ (noting that $\hat{\rho}(s)\left(y^s-Z_2(s)y^{2n-s}\right)$ is analytic on the strip $n\leq \Re(s)\leq \eta$), and for the last equality we used Mellin inversion formula for $\rho$. Comparing $(\ref{lhsado})$ and the above equations completes the proof.
\end{proof}
%\begin{proof}[Proof of Proposition \ref{prop:ef}]
%In view of Corollary \ref{even}, it remains to prove the case when $p$ and $q$ are both odd. By Lemma \ref{prelem1} and \ref{oddcase}, we know that Proposition \ref{prop:ef} holds for any functions in $\mathcal A^{1,1}$. Then by Lemma \ref{prelem2}, we know that for any pair of positive odd integers $(p,q)$ Proposition \ref{prop:ef} holds for any functions in $\mathcal A^{p,q}$.
%\end{proof}

%\section{Proof of second moment formula}
%We now have almost all the ingredients needed to proof the second moment formula of Theorem \ref{thm:Rogersformula}.
\subsection{The isometry}
Before completing the proof of the second moment formula, we need to define the isometry of $L^2_{\rm even}(\R^{2n})$ which is the same as $L^2_{\rm even}(Q_n\bk G_n)$.
For any even bounded compactly supported $f$ on $Q_n\bk G_n$, we define $\iota(f)$
by
\begin{equation}\label{iotadef}
\iota(f)=\frac12\mathcal P_f-f-\frac{1}{2\zeta(2n)}\int_{\R^{2n}}f(\vec{x})d\vec{x}.
\end{equation}
We now show that this is indeed an isometry.
\begin{Prop}\label{p:isometry}
The map $\iota:L^2_{\rm even}(Q_n\bk G_n)\to L^2_{\rm even}(Q_n\bk G_n)$ is an isometry.
\end{Prop}
\begin{proof}
We first show that for any $(p,q)$ with same parity and any $f\in \cA^{p,q}$ we have $\|\iota(f)\|_2=\|f\|_2$.
Indeed, writing $f(a_yk)=\rho(y)\phi(k)$  we have
\begin{align*}
||f||_2^2&=\frac{2\pi^n}{\Gamma(n)}\int_0^{\infty}|\rho(y)|^2\frac{dy}{y^{2n+1}}\|\phi\|_2^2\\
&=\frac{2\pi^n}{\Gamma(n)}\frac{1}{2\pi i}\int_{(n)}|\hat{\rho}(s)|^2ds \|\phi\|_2^2,
\end{align*}
where the first line is just integration in polar coordinates and the second is Plancherel's theorem.
Now by Proposition \ref{intformula-s}, Proposition \ref{prop:ef} and the relation $(\ref{haarrelation})$  we have
\begin{align*}
\iota(f)(a_yk)&=\frac{(-1)^p}{2\pi i}\int_{(n)}\hat{\rho}(s)Z_{p+q}(s)y^{2n-s}ds\phi(k)\\
&=\frac{(-1)^p}{2\pi i}\int_{(n)}\hat{\rho}(2n-s)Z_{p+q}(2n-s)y^{s}ds\phi(k)=v(y)\phi(k),
\end{align*}
where $v(y)$ has Mellin transform $\hat{v}(s)=(-1)^p\hat{\rho}(2n-s)Z_{p+q}(2n-s)$ for $1<\Re(s)<2n-1$. In particular, by \eqref{fees} for $\Re(s)=n$ we have that
$|\hat{v}(s)|=|\hat{\rho}(2n-s)|$ and hence
\begin{align*}
\|\iota(f)\|^2&=\frac{2\pi^n}{\G(n)}\frac{1}{2\pi i}\int_{(n)}\left|\hat{\rho}(2n-s)\right|^2ds\|\phi\|_2^2\\
&=\frac{2\pi^n}{\G(n)}\frac{1}{2\pi i}\int_{(n)}\left|\hat{\rho}(s)\right|^2ds\|\phi\|_2^2=\|f\|^2.
\end{align*}

Now since the different spaces $\cA^{p,q}$ are orthogonal, and any even smooth compactly supported functions can be decomposed as
$f=\sum_{p\equiv q\pmod{2}} f_{p,q}$ with $f_{p,q}\in \cA^{p,q}$, we get that $\|\iota(f)\|_2=\|f\|_2$ for all even, smooth, compactly supported functions. Since even, smooth, compactly supported functions are dense in $L^2_{\rm even}(\R^{2n})$ this concludes the proof.
\end{proof}

%\begin{rmk}
%\color{red}{From our formula, we can see that $\int\bar{f}\iota(f)d\vec{x}$ is real, in fact we can deduce this directly using the formula we have for $\iota(f)$. For simplicity, assume $f=\rho\phi\in \cA^{p,q}$, then
%\begin{align*}
%\int\bar{f}\iota(f)d\vec{x}&=\frac{2\pi^n}{\G(n)}\frac{(-1)^p}{2\pi i}\|\phi\|_2^2\int_{(n)}\overline{\hat{\rho}(s)}\hat{\rho}(2n-s)Z_{p+q}(2n-s)ds\\
%&=\frac{2\pi^n}{\G(n)}\frac{(-1)^p}{2\pi i}\|\phi\|_2^2\left(\int_{(n)}^+\hat{\bar{\rho}}(\bar{s})\hat{\rho}(\bar{s})Z_{p+q}(\bar{s})ds+\int_{(n)}^+\hat{\bar{\rho}}(s)\hat{\rho}(s)Z_{p+q}(s)ds\right).\\
%\end{align*}
%But $\overline{\hat{\bar{\rho}}(\bar{s})\hat{\rho}(\bar{s})Z_{p+q}(\bar{s})}=\hat{\rho}(s)\hat{\bar{\rho}}(s)Z_{p+q}(s)$. Hence
%$$\int\bar{f}\iota(f)d\vec{x}=\frac{2\pi^n}{\G(n)}\frac{(-1)^p}{2\pi i}\|\phi\|_2^2\int_{(n)}^+2\Re(\hat{\bar{\rho}}(\bar{s})\hat{\rho}(\bar{s})Z_{p+q}(\bar{s}))ds$$
%is real.
% }
%\end{rmk}

\subsection{Proof of second moment formula}
Let $f$ be an even, compactly supported function on $\R^{2n}$ and let $\tilde f$ be the corresponding function on $Q_n\bk G_n$ given by $\tilde{f}(g)=f(\vec{e}_{2n}g)$.
Using Proposition \ref{p:F2Theta}, Proposition \ref{lem:pi2} and the relation \eqref{haarrelation} we see that
$$\|F_f\|_2^2=\|\Theta_{\tilde f}\|_2^2=\frac{1}{\zeta(2n)}\int_{\R^{2n}}\overline{f(\vec{x})}\mathcal P_f(\vec{x})d\vec{x},$$
Now from the definition of the isometry $\iota$ we have
$$\cP_f=2\iota(f)+2f+\frac{1}{\zeta(2n)}\int_{\R^{2n}}f(\vec{x})d\vec{x},$$
and plugging this in we get that
\begin{align*}
\|F_f\|_2^2&=\frac{1}{\zeta(2n)}\int_{\R^{2n}}\overline{f(\vec{x})}\left(2\iota(f)(\vec{x})+2f(\vec{x})+\frac{1}{\zeta(2n)}\int_{\R^{2n}}f(\vec{x})d\vec{x}\right)d\vec{x}\\
&=\left|\frac{1}{\zeta(2n)}\int_{\R^{2n}}f(\vec{x})d\vec{x}\right|^2+\frac{2}{\zeta(2n)}\int_{\R^{2n}}|f(\vec{x})|^2d\vec{x}+\frac{2}{\zeta(2n)}\int_{\R^{2n}}\overline{f(\vec{x})}\iota(f)(\vec{x})d\vec{x},
\end{align*}
thus concluding the proof.
 \qed

 \begin{rem}
For an even bounded non-negative function $f\in L^2(\R^{2n})$ that is not compactly supported, we can take a sequence $f_j(\vec{x})=f(\vec{x})\chi_{B_j}(\vec{x})$ with $B_j\subseteq \R^{2n}$ the ball of radius $j$. Then $f_j$ monotonously converges to $f$ pointwise, as well as in $L^2$. Moreover $|F_{f_j}(\Lambda)|^2$ also monotonously converges to $|F_f(\Lambda)|^2$, hence, by monotone convergence $\int_{Y_n}|F_f(\Lambda)|^2d\mu_n=\lim_{j\to\infty} \int_{Y_n}|F_{f_j}(\Lambda)|^2d\mu_n$ giving the same formula for $f$.
\end{rem}

\begin{rem}
Given two even compactly supported functions, $f,g$ by computing the mean square of $F_{f+g}$  the second moment formula is equivalent to the following inner product formula
\begin{equation}\label{e:inner}
\<F_{f},F_g\>_{Y_n}=\frac{\<f,1\>\<1,g\>}{\zeta(2n)^2}+\frac{2}{\zeta(2n)}(\<f,g\>+\<\iota(f),g\>),\end{equation}
where $\<\cdot,\cdot\>$ and  $\<\cdot,\cdot\>_{Y_n}$ are the corresponding inner products on $L^2(\R^{2n})$ and $L^2(Y_n,\mu_n)$.
\end{rem}

\section{Applications to counting}

In this section, we apply our second moment formula to get results on lattice point counting problems for a generic symplectic lattice.
\subsection{Mean square bound}
Our first simple application gives a mean square bound for the discrepancy. For the primitive lattice points this is almost immediate while for all regular lattice points we follow the standard argument converting results from primitive lattice points to all lattice points.
\begin{proof}[Proof of Theorem \ref{thm:meansquare}]
%\begin{Cor}\label{cor:counting}
%For any Borel set $B\subset (\R^{2n})^{\times}=Q_n\bk G_n$ of finite volume, we have
%\begin{equation}\label{countingestiamte}
%\int_{Y_n}\left|\#(\Lambda_{\rm pr}\cap B)-\frac{\vol(B)}{\zeta(2n)}\right|^2d\mu_n(\Lambda)\leq C_n\rm{vol}(B),
%\end{equation}
%where $C_n=\frac{4}{\zeta(2n)}$.
%\end{Cor}
%\begin{proof}
For any Borel set $B\subseteq \dot{\R}^{2n}$ let $\chi_B$ be the characteristic function of $B$ and let $f(\vec{x})=\frac{\chi_B(\vec{x})+\chi_B(-\vec{x})}{2}\in L^2_{\rm even}(\R^{2n})$ its even part so that $\int_{\R^{2n}}f(\vec{x})d\vec{x}=\vol(B)$ and $\|f\|_2^2\leq \vol(B)$. Note that for any $\Lambda\in Y_n$, $F_f(\Lambda)=F_{\chi_B}(\Lambda)=\#(\Lambda_{\rm pr}\cap B)$ and by \eqref{firstmomentin}
$$\int_{Y_n}F_f(\Lambda)d\mu_n(\Lambda)=\frac{1}{\zeta(2n)}\int_{\R^{2n}}f(\vec{x})d\vec{x}=\frac{\textrm{vol}(B)}{\zeta(2n)}.$$
Thus by Theorem \ref{thm:Rogersformula} and Cauchy-Schwartz we have
\begin{align*}
\int_{Y_n}\left|\#(\Lambda_{\rm pr}\cap B)-\frac{\vol(B)}{\zeta(2n)}\right|^2d\mu_n(\Lambda)&=\int_{Y_n}\left|F_f(\Lambda)-\frac{\textrm{vol}(B)}{\zeta(2n)}\right|^2d\mu_n(\Lambda)\\
&=\int_{Y_n}\left|F_f(\Lambda)\right|^2d\mu_n(\Lambda)-\left(\frac{\textrm{vol}(B)}{\zeta(2n)}\right)^2\\
&=\frac{2}{\zeta(2n)}\int_{\R^{2n}}\left(\left|f(\vec{x})\right|^2+\overline{f(\vec{x})}\iota(f)(\vec{x})\right)d\vec{x}\\
%&\leq \frac{4}{\zeta(2n)}\int_{\R^{2n}}\left|f^{\textrm{even}}(x)\right|^2dx\\
%&\leq\frac{2}{\zeta(2n)}\left(\textrm{vol}(B)+\int_{\R^{2n}}\chi_B(x)\chi_B(-x)dx\right)\\
&\leq\frac{2}{\zeta(2n)}\left(\|f\|_2^2+\|f\|_2\|\iota(f)\|_2\right)\\
&=\frac{4\|f\|_2^2}{\zeta(2n)}\leq \frac{4\vol(B)}{\zeta(2n)},
\end{align*}
%\begin{equation}\label{countingestiamte}
%\int_{Y_n}\left|\#(\Lambda_{\rm pr}\cap B)-\frac{\vol(B)}{\zeta(2n)}\right|^2d\mu_n(\Lambda)\leq \frac{4\vol(B)}{\zeta(2n)},
%\end{equation}
which gives \eqref{eq:measureestimate} after multiplying both sides by $\frac{\zeta(2n)^2}{\vol(B)^2}$.

%$$\|f\|_2^2=\int_{\R^{2n}}\left|\frac{\chi_B(\vec{x})+\chi_B(-\vec{x})}{2}\right|^2d\vec{x}=\frac{1}{2}\int_{\R^{2n}}\left(\chi_B(\vec{x})+\chi_B(\vec{x})\chi_B(-\vec{x})\right)d\vec{x}\leq \textrm{vol}(B),$$
%where for the second equality we used $\chi_B^2=\chi_B$ and for the inequality we used $\chi_B(-\vec{x})\leq 1$. This finishes the proof of \eqref{}.

Next for the regular lattice point counting problem, consider the dilated functions $f^k(\vec{x})=f(k \vec{x})$ and $\chi_B^k(\vec{x})=\chi_B(k\vec{x})$.
Since, by our assumption, $0\not\in B$ we have that
$$\#(\Lambda\cap B)=\sum_{k=1}^\infty \#(k \Lambda_{\rm pr}\cap B)=\sum_{k=1}^\infty F_{f^k}(\Lambda).$$
 Integrating over $Y_n$ and using  \eqref{firstmomentin} we get that
\begin{align*}
\int_{Y_n}\#(\Lambda\cap B)d\mu_n(\Lambda)&=\sum_{k=1}^\infty\int_{Y_n}F_{f^k}(\Lambda)d\mu_n(\Lambda)
=\frac{1}{\zeta(2n)}\sum_{k=1}^\infty \int_{\R^{2n}}f^k(\vec{x})d\vec{x}\\
&=\frac{1}{\zeta(2n)}\sum_{k=1}^\infty\frac{1}{k^{2n}} \int_{\R^{2n}}f(\vec{x})d\vec{x}=\vol(B).
\end{align*}
Next, using the moment formula in the form of \eqref{e:inner}, gives
\begin{align*}
\int_{Y_n}|\#(\Lambda\cap B)|^2d\mu_n(\Lambda)&=\sum_{k,l}\langle F_{f^k},F_{f^l}\rangle_{Y_n}\\
&= \sum_{k,l}\frac{\<f^k,1\>\<1,f^l\>}{\zeta(2n)^2}+\sum_{k,l}\frac{2}{\zeta(2n)}\left(\<f^k,f^l\>+\<\iota(f^k),f^l\>\right)\\
 &=\vol(B)^2+\frac{2}{\zeta(2n)}\sum_{k,l}\left(\<f^k,f^l\>+\<\iota(f^k),f^l\>\right),
\end{align*}
and hence
\begin{align*}
\int_{Y_n}|\#(\Lambda\cap B)-\vol(B)|^2d\mu_n(\Lambda)&=\int_{Y_n}|\#(\Lambda\cap B)|^2d\mu_n(\Lambda)-\vol(B)^2 \\
 &=\sum_{k,l}\frac{2}{\zeta(2n)}\left(\<f^k,f^l\>+\<\iota(f^k),f^l\>\right)\\
 &\leq\frac{2}{\zeta(2n)}\sum_{k,l}\left(\|f^k\|_2\|f^l\|_2+\|\iota(f^k)\|_2\|f^l\|_2\right)\\
 &=\frac{4 \|f\|_2^2}{\zeta(2n)}\sum_{k,l}\frac{1}{k^nl^n} \leq \frac{4\zeta(n)^2\vol(B)}{\zeta(2n)}.
 \end{align*}
 Dividing both sides by $\vol(B)^2$ concludes the proof of \eqref{eq:measureestimate2}.
\end{proof}
\subsection{Schmidt's argument}
We can now use the above mean square estimate together with Schmidt's argument from \cite{Schmidt1960} to prove Theorem \ref{thm:counting}, which  follows from the following by taking $\psi(x)=c/x^2$ below with an appropriate choice of constant $c$.
\begin{Thm}\label{countingg}
Let $\cB$ be a linearly ordered  family of Borel sets in $\dot{\R}^{2n}$. Let $\psi$ be a positive, non-increasing function such that $e^t\psi(t)$ is eventually non-decreasing and $\int_1^\infty \psi(t)dt<\infty$. Then for $\mu_n$-a.e. $\Lambda\in Y_n$ there is $C_\Lambda$ such that for all $B\in \cB$ with $\vol(B)>C_\Lambda$
$$|\#(\Lambda_{\rm pr}\cap B)-\frac{\vol(B)}{\zeta(2n)}|\leq \sqrt{\vol(B)}\frac{\log(\vol(B)) }{\psi^{1/2}(\log\vol(B))},$$
and
$$|\#(\Lambda\cap B)-\vol(B)|\leq \sqrt{\vol(B)}\frac{\log(\vol(B)) }{\psi^{1/2}(\log\vol(B))}.$$
\end{Thm}
\begin{proof}
The arguments are identical to the ones given in \cite{Schmidt1960}, and we include the details for the readers' convenience.
Since the proofs for the primitive lattice point counting and for the regular lattice point counting follow from the exact same argument, we will give the details only for first one.

First note that if the set of volumes of sets in $\cB$ is bounded the statement holds vacuously by taking $C_{\Lambda}$ larger than the volume of any set in $\cB$, so we can assume there are arbitrarily large volumes. With this assumption, by  \cite[Lemma 1]{Schmidt1960}, we can assume without loss of generality (after perhaps adding more sets to $\cB$) that $\{\vol(B)\ |\ B\in\cB\}=\R^+$. Thus for any positive integer $N$, there exists some $B_N\in\cB$ with $\vol(B_N)= N$. For any $\Lambda\in Y_n$, and $N\geq 1$ we denote
$$S_N(\Lambda)=\#(\Lambda_{\rm pr}\cap B_{N})-\frac{N}{\zeta(2n)}
$$
and for any $1\leq N_1< N_2$
$${}_{N_1}S_{N_2}(\Lambda)=\#(\Lambda_{\rm pr}\cap (B_{N_2}\bk B_{N_1}))-\frac{N_2-N_1}{\zeta(2n)}.
$$
For any integer $T\geq 3$ we denote by  $\cK_T$ the set of all pairs of integers $N_1, N_2$ of the form $0\leq N_1< N_2\leq 2^T, N_1=\ell2^t$ and $N_2=(\ell+1)2^t$, for integers $\ell$ and $t\geq 0$. Applying \eqref{eq:measureestimate} to the sets $B_{N_2}\setminus B_{N_1}$ and repeating the exact same arguments as in
 \cite[Lemma 2]{Schmidt1960}  we get that for any $T\geq 3$
\begin{equation}\label{lemestimate}
\sum_{(N_1,N_2)\in \mathcal K_T}\int_{Y_n}{}|_{N_1}S_{N_2}(\Lambda)|^2d\mu_n(\Lambda)\leq C_n(T+1)2^T,
\end{equation}
where $C_n=\frac{4}{\zeta(2n)}$. Next, let $\cE_T\subseteq Y_n$ denote the set of all lattices $\Lambda\in Y_n$ for which
\begin{equation}\label{defft}
\sum_{(N_1,N_2)\in \mathcal K_T}{}|_{N_1}S_{N_2}(\Lambda)|^2> \frac{(T+1)2^T}{800\psi(\log(2)(T-1))}.
\end{equation}
Then \eqref{lemestimate} implies that
\begin{equation}\label{measureestimate}
\mu_n(\cE_T)<  800C_n\psi(\log(2)(T-1)).\end{equation}
Consider the limsup set
$$\cE_\infty=\limsup_{T\to\infty}\cE_T:=\bigcap_{j\geq 3}\bigcup_{T\geq j}\cE_T.$$
Since the right-hand side of \eqref{measureestimate} is summable we have that $\mu_n(\cE_\infty)=0$, and we will take its complement $Y_n\setminus \cE_\infty$ to be the full measure set of lattices for which the discrepancy is small.

Now, note that for $N\leq 2^T$, the interval $[0,N)$ can be expressed as a disjoint union of at most $T$ intervals of the form $[N_1,N_2)$ with $(N_1,N_2)\in \mathcal K_T$.
We can thus write
$$S_N(\Lambda)=\sum_{[N_1,N_2)\in\mathcal I}{}_{N_1}S_{N_2}(\Lambda),$$
where $\mathcal I$ is a set consisting of at most $T$ intervals of the form $[N_1,N_2)$ with $(N_1,N_2)\in\mathcal K_T$. Using Cauchy-Schwartz and $(\ref{defft})$ we have for any $\Lambda\notin \cE_T$ and any $N<2^T$
\begin{displaymath}
|S_N(\Lambda)|^2\leq \frac{(T+1)^22^T}{800\psi(\log 2(T-1))}.
\end{displaymath}
Now, for any $\Lambda\not\in \cE_\infty$ there is some $T_\Lambda$ such that for all $T\geq T_\Lambda$ we have that $\Lambda\not\in \cE_T$ and hence   $|S_N(\Lambda)|^2\leq \frac{(T+1)^22^T}{100\psi((T-1)\log 2)}$  for all $N<2^T$.

Now, for any $\Lambda\not\in \cE_\infty$ let $C_\Lambda= \max\{2^{T_\Lambda}+1,N_0\}$ with $N_0$ sufficiently large that for all $N\geq N_0$ we have
\begin{equation} \label{e:N0}
\frac{\sqrt{N+1}\log(N+1)}{2\psi^{1/2}(\log(N+1))}+1\leq \frac{\sqrt{N}\log(N)}{\psi^{1/2}(\log(N))},\end{equation}
where we used that $N\psi(\log(N))$ is eventually non-decreasing to make sure such $N_0$ exists.
Then, for any integer $N> C_{\Lambda}-1$, choose integer $T$ such that $2^{T-1}\leq N<2^T$. In particular we have that $T\geq T_{\Lambda}$ and $N<2^T$ so,
\begin{align*}
|S_N(\Lambda)|^2&\leq \frac{(T+1)^2 2^T}{800 \psi(\log 2(T-1))}\\
&\leq \left(\frac{\log N}{\log 2}+2\right)^2\ \frac{N}{400\psi(\log(N))}<\frac{N\log^2 N}{4\psi(\log(N))} ,
\end{align*}
where we used that $\left(\frac{\log N}{\log 2}+2\right)\leq 10\log(N)$ for all $N\geq 2$.
We have thus verified that for all $N>C_\Lambda-1$ we have $|S_N(\Lambda)|\leq\frac{ \sqrt{N}\log(N)}{2\psi^{1/2}(\log(N))}$.

Next, for any set $B\in \cB$ with $\vol(B)> C_{\Lambda}$, there exists an integer $N> C_{\Lambda}-1$ such that $B_{N}\subseteq B\subseteq B_{N+1}$.
We can interpolate, to bound
$$\left|\#(\Lambda_{\rm pr}\cap B)-\frac{\vol(B)}{\zeta(2n)}\right|\leq\max\left\{\left|S_{N}(\Lambda)\right|, \left|S_{N+1}(\Lambda)\right|\right\}+1,$$
and since $N,N+1\geq C_\Lambda-1$ we can bound
$$\left|\#(\Lambda_{\rm pr}\cap B)-\frac{\vol(B)}{\zeta(2n)}\right|\leq \frac{\sqrt{N+1}\log(N+1)}{2\psi^{1/2}(\log(N+1))}+1\leq \frac{\sqrt{\vol(B)}\log(\vol(B))}{\psi^{1/2}(\log(\vol(B)))} ,$$
where we used \eqref{e:N0} recalling that $N\geq N_0$.

The same proof with the obvious modifications give the same result for the general lattice point counting problem.
\end{proof}

\subsection{Dilations}
We now want to apply our result for the special case where our family is given by a dilation of a fixed set $B\subseteq \dot{\R}^{2n}$.
\begin{proof}[Proof of theorem \ref{t:dilation}]
Write $B=\bigsqcup_{j=1}^k B_j$ with $B_j=B_j^+\setminus B_j^-$ and note that any dilation is of the form $tB=\bigsqcup_{j=1}^k tB_j$ and that $tB_j=tB_j^+\setminus tB_j^-$.
Moreover,
$$\#(\Lambda\cap tB)=\sum_{j=1}^k\#(\Lambda\cap tB_j^+)-\sum_{j=1}^k\#(\Lambda\cap tB_j^-),$$
and similarly for the primitive lattice points.
Considering the finitely many linearly ordered families $\cB_j^{\pm}=\{tB_j^\pm:t\in \R^+\}$ and applying Theorem \ref{countingg} to each one with $\psi(t)=\frac{64n^4k^2}{t^2}$, we get that, for each $j$, for $\mu_n$-a.e. $\Lambda\in Y_n$ there is $C_{\Lambda,j}^{\pm}$ such that for all $t>C_{\Lambda,j}^{\pm}$
$$| \#(\Lambda\cap tB_j^\pm)-\vol(tB_j^{\pm}) |\leq \frac{t^n\log^2(t)}{2k}.$$
The intersection of these finitely many full measure sets is still of full measure and taking $C_{\Lambda}=\max\{ C_{\Lambda,j}^{\pm}: 1\leq j\leq k\}$ we get that for all $t>C_\Lambda$
\begin{align*}
|\#(\Lambda\cap tB)-t^{2n}|%&=\left|\sum_{j=1}^k(\#(\Lambda\cap tB_j^+)-\vol(tB_j^+))-\sum_{j=1}^k(\#(\Lambda\cap tB_j^-)-\vol(tB_j^-))\right|\\
&\leq \sum_{j=1}^k\left|\#(\Lambda\cap tB_j^+)-\vol(tB_j^+)\right|+\left|\#(\Lambda\cap tB_j^-)-\vol(tB_j^-)\right|\leq t^n\log^2(t),
\end{align*}
so that $D(\Lambda,tB)\leq \frac{\log^2(t)}{t^n}$ as claimed.
A similar argument gives the same bound for the primitive lattice points.
\end{proof}

\bibliographystyle{alpha}
\bibliography{DKbibliog}

\end{document}